\documentclass[english]{smfart}

\usepackage{hyperref}

\usepackage{setspace}
\usepackage{amsmath}
\usepackage{smfthm}
\usepackage{amsfonts}
\usepackage{amssymb}
\usepackage[utf8]{inputenc}
\usepackage[margin=2.8cm]{geometry} 
\usepackage{graphicx} 
\usepackage[T1]{fontenc}
\usepackage[utf8]{inputenc}
\usepackage{xcolor}
\usepackage[english]{babel}

\usepackage{tikz}

\theoremstyle{plain}
\newtheorem{theorem}{Theorem}
{\bfseries}{\itshape}
{\bfseries}{\itshape}
\newtheorem{proposition}[theorem]{Proposition}

\newtheorem{lemma}[theorem]{Lemma}

\theoremstyle{definition}
\newtheorem{definition}[theorem]{Definition}
\newtheorem{example}[theorem]{Example}

\theoremstyle{remark}
\newtheorem{remark}[theorem]{Remark}

\usepackage{mathrsfs}
\usepackage{enumerate}
\usepackage{tikz}

\newcommand{\C}{\mathcal{C}}
\newcommand{\A}{\mathcal{A}}

\newcommand{\E}{\mathcal{E}}

\newcommand{\Z}{\mathbb{Z}}
\newcommand{\F}{\mathbb{F}}
\newcommand{\N}{\mathbb{N}}

\newcommand{\Fq}{\mathbb{F}_{q}}
\newcommand{\LL}{\mathbb{L}}
\newcommand{\K}{\mathbb{K}}
\newcommand{\EE}{\mathbb{E}}

\DeclareMathOperator{\Gal}{Gal}
\DeclareMathOperator{\GL}{GL}

\renewcommand{\K}{\mathbb{K}}
\newcommand{\Kbar}{\overline{\mathbb{K}}}
\renewcommand{\A}{\mathbb{A}}
\renewcommand{\P}{\mathbb{P}}
\renewcommand{\C}{\mathbb{C}}
\renewcommand{\O}{\mathcal{O}}
\newcommand{\Q}{\mathbb{Q}}
\newcommand{\R}{\mathbb{R}}

\newcommand{\curve}[1]{\ensuremath{\mathscr{#1}}}
\def\E{\curve{E}}
\def\X{\curve{X}}
\def\Y{\curve{Y}}
\newcommand{\Lcal}{\curve{L}}

\newcommand{\Pic}{\text{Pic}}
\newcommand{\Div}{\text{Div}}

\renewcommand{\leq}{\leqslant} 
\renewcommand{\geq}{\geqslant}

\newcommand{\code}[1]{\ensuremath{\mathcal{#1}}}

\newcommand{\CC}{\code{C}}

\newcommand{\CL}{\CC_L}
\newcommand{\cP}{\mathcal{P}}
\newcommand{\CLdef}{\CC_L(\X, \cP, G)}

\newcommand{\word}[1]{\ensuremath{\boldsymbol{#1}}}

\newcommand{\cv}{\word{c}}

\newcommand{\ev}{\word{e}}

\newcommand{\xv}{{\word{x}}}
\newcommand{\yv}{{\word{y}}}

\renewcommand{\dh}{\ensuremath{\text{d}_{\text{H}}}}
\newcommand{\wh}{\ensuremath{\text{w}_{\text{H}}}}
\newcommand{\card}[1]{\sharp{ #1 }}

\newcommand{\map}[4]{\left\{
\begin{array}{ccc}
#1 & \longrightarrow & #2 \\ #3 & \longmapsto & #4 
\end{array}
\right.}

\newcommand{\ratmap}[4]{\left\{
\begin{array}{ccc}
#1 & \dashrightarrow & #2 \\ #3 & \longmapsto & #4 
\end{array}
\right.}

\newcommand{\eqdef}{\stackrel{\textbf{def}}{=}}
\renewcommand{\span}{\text{Span}}
\newcommand{\SL}{\mathbf{SL}_2(\Z)}
\newcommand{\SLl}{\mathbf{SL}_2(\F_{\ell})}
\renewcommand{\GL}{\mathbf{GL}}
\newcommand{\HP}{\mathbb{H}}
\newcommand{\ie}{{\em i.e.}}

\theoremstyle{definition}
\newtheorem{exercise}[theorem]{Exercise} 

\title{Codes and modular curves}
\author{Alain Couvreur}
\address{Inria, France}
\email{alain.couvreur@inria.fr}

\begin{document}

\begin{abstract}
  These lecture notes have been written for a course at the
  {\em Algebraic Coding Theory} (ACT) summer school 2022 that took place in
  the university of Zurich.
The objective of the course propose an in--depth presentation of the proof of
one of the most striking results of coding theory: Tsfasman
Vl\u{a}du\c{t} Zink Theorem, which asserts that for some prime power
$q$, there exist sequences of codes over $\Fq$ whose asymptotic
parameters beat random codes.
\end{abstract}

\maketitle

\tableofcontents

\section*{Introduction}
Algebraic Geometry (AG) codes is a particularly exciting topic lying
at the intersection between number theory, algebraic geometry and
coding theory.  The birth of this research area dates back to the
early 80's with the introduction by Goppa \cite{goppa1981dansssr} of a
new family of codes obtained by evaluating residues of some
differential forms on a given curve. Quickly after, Tsfasman,
Vl\u{a}du\c{t}, Zink \cite{tsfasman1982mn} and independently Ihara
\cite{ihara1981jfsuTokyo} proved the existence of sequences of modular
curves and Shimura curves having an excellent asymptotic ratio number
of points {\em v.s.} genus. An immediate but extremely striking
corollary is the existence of sequences of codes beating the Gilbert
Varshamov bound, in short: codes better than random codes.  This
remarkable and totally unexpected result turned out to be the first
stone of the development of a whole theory: that of AG codes.
Surprisingly, a very comparable breakthrough happened in graph theory
the late 80's.  Indeed, in 1988, Lubotsky, Philips and Sarnak
\cite{LPS88} and independently Margulis \cite{margulis1988} used
Cayley graphs on quotients of $\SL$ to prove the existence of a family
of graphs whose girth, {\em i.e.} the length of their shortest cycle,
exceeds the girth obtained with the probabilistic method. In both
situations, coding theory and graph theory, the use of elegant
algebraic structures unexpectedly beat random constructions.

The objective of this lecture is to present in an (almost)
self-contained presentation, the beginning of this wonderful story:
the original proof of Tsfasman, Vl\u{a}du\c{t} and Zink Theorem.  It
should be mentioned that in 1995, Garcia and Stichtenoth \cite{GS1995}
proposed another and somehow more explicit approach to design sequence
of curves (actually function fields but the two objects are
equivalent) reaching the so-called Drinfeld--Vl\u{a}du\c{t}
\cite{vladut1983faa}. It could be considered as strange to present the
original proof which turns out to be much more complicated than Garcia
and Stichtenoth's one but there are some reasonable motivations for
that:
\begin{itemize}
\item[\textbullet] Tsfasman, Vl\u{a}du\c{t} and Zink's proof testifies from the
  richness of the theory of algebraic geometry codes, with a proof
  involving deep results from algebraic geometry and number theory.
\item[\textbullet] This original proof is frequently cited while few references give a
  complete presentation of it and (in my personal opinion), none of the
  papers of Tsfasman {\em et. al.} and Ihara provide an enough detailed proof.
  In both articles, the proof is made of less than ten lines hiding a huge
  amount of prerequisites.
\item[\textbullet] Finally, I wished to give that lecture, because this proof is
  beautiful and elegant and even if I am not among the mathematicians
  who do maths {\em pour la beaut\'e de la chose}\footnote{Litterally
    : ``for the beauty of the thing''} it is sometimes pleasant to take the time
  to appreciate the elegance of some development.
\end{itemize}

\subsection*{Outline of these notes}
We start in Section~\ref{sec:codes} with bases on linear codes and
their asymptotic behaviour. Section~\ref{sec:curves} gives an
introduction to algebraic curves by providing the necessary material
in algebraic geometry.  Section~\ref{sec:ag_codes} introduces
algebraic geometry codes and states the main result:
Tsfasman--Vl\u{a}du\c{t}--Zink Theorem. The remainder of the notes are
dedicated to the proof of this statement. Sections~\ref{sec:elliptic}
and~\ref{sec:modular} provide further material on elliptic and modular
curves respectively. Section~\ref{sec:proof_of_main} concludes the
proof.

\subsection*{Acknowledgements}
First, I would like to thank Gianira Alfarano, Karan Khaturia,
Alessandro Neri, Violetta Weger, the organisers of the {\em Algebraic
  Coding Theory Summer
  School}\footnote{\url{https://math.uzh.ch/act/}} 2022 who gave me
the motivation to type-write old hand-written notes. I would probably
never have found the time to do it if they did not ask me for. Several
colleagues spent time to carefully read these notes. In particular, I
express a deep gratitude to Elena Berardini, Maxime Bombar, Grégoire
Lecerf, Jade Nardi, Christophe Ritzenthaler, Joachim Rosenthal and
Gilles Z\'emor for their relevant comments on the preliminary version
of the notes.

The author is funded by the french {\em Agence nationale de la recherche}
for the collaborative project ANR-21-CE39-0009-BARRACUDA.

 \section{Linear Codes}\label{sec:codes}

\subsection{Context}
In the sequel we are interested in {\em linear $q$--ary codes}, which
are linear subspaces of $\Fq^n$. What makes the study hard, but also
deeply interesting is that we are not only considering elementary
objects such as finite dimensional vector spaces but
spaces endowed with a {\bf metric}: the {\em Hamming metric}. The
Hamming distance between two vectors $\xv, \yv \in \Fq^n$ is denoted
by
\[
  \dh (\xv, \yv) \eqdef \card{ \left\{i \in \{1, \dots, n\} ~|~ x_i
      \neq y_i \right\}}.
\]
The {\em Hamming weight} of a vector is its Hamming distance to the
zero vector.
\[
  \forall \xv \in \Fq^n,\qquad \wh(\xv) \eqdef \dh (\xv, \mathbf{0}).
\]

\subsection{Linear codes}\label{ss:linear_codes}
Unless otherwise specified, a {\em code} will denote a linear subspace
$\CC \subseteq \Fq^n$. The vectors of $\CC$ are usually referred to as
{\em codewords}. The {\em dimension} of $\CC$ regarded as an
$\Fq$--vector space is always denoted by $k$ and its {\em minimum
  distance} denoted by $d$ is defined as
\[
  d \eqdef \min_{\stackrel{\xv, \yv \in \CC}{\xv \neq \yv}} \left\{\dh
    (\xv, \yv) \right\} = \min_{\cv \in \CC \setminus \{0\}} \left\{
    \wh(\cv) \right\},
\]
where the last equality is a consequence of the linearity.  The {\em
  parameters} of a code $\CC \subseteq \Fq^n$ refer to the triple
$n, k, d$ and is usually denoted as $[n,k,d]_q$, where the cardinality
$q$ of the base field is recalled in subscript.  Finally, one can also
be interested in the {\em rate} and {\em relative distance} of a code,
respectively defined and denoted as follows:
\[
  R \eqdef \frac k n \qquad \text{and}\qquad \delta \eqdef \frac d n\cdot
\]
A longstanding problem in coding theory is which kind of triples of
parameters $[n,k,d]$ can be achieved? A code will be considered as
``good'' if both $k$ and $d$ are as close as possible to
$n$. However, many upper bounds exist, the most elementary one being the
{\em Singleton bound} saying that for any code with parameters $[n,k,d]_q$
we have
\begin{equation}\label{eq:Singleton}
  k+d \leq n+1.
\end{equation}
The rationale behind this question is that both $k$ and $d$ quantify
some feature of linear codes.  Suppose we are given a {\em
  transmission channel}, that can be either a wire or a wireless
communication for instance an exchange between electronic devices
like between a computer and a WiFi antenna. The rate is nothing but
the ratio of information divided by the quantity of data which is
actually sent across the channel.  Hence, the rate $R = k/n$
quantifies the efficiency of encoding.

On the other hand, the minimum distance quantifies how far are words
from each other and hence the theoretical ability to recover an original
message from a corrupted codeword\footnote{Here we do not introduce
any consideration about practical algorithms to correct errors}.

Finally, suppose that our objective is to correct errors from a given
channel. Consider for instance the {\em
  $q$--ary symmetric channel with parameter $p \in [0,1 - \frac 1 q]$}
which takes as input a vector $\cv \in \Fq^n$ and outputs the vector
$\cv + \ev$ where $\ev = (e_1, \dots, e_n)$ and the $e_i$'s are
independent random variables over $\Fq$ taking value $0$ with
probability $1-p$ and any other value in $\Fq \setminus \{0\}$ with probability
$\frac{p}{q-1}$.  The average weight of our error vector satisfies
\[
  \mathbb{E}(\wh (\ev)) = pn.
\]
However, for small values of $n$, deviations may happen and it is
possible that our input vector $\cv$ is corrupted by much more than
$\lfloor pn \rfloor$ errors. Therefore, it is relevant to consider
large values of $n$ for which the law of large numbers will
assert us that the weight of the error will be close to its expectation.

This last discussion motivates the search of sequences of codes
${(\CC_s)}_{s \in \N}$ with parameters $[n_s, k_s, d_s]$ where
\[
  \lim_{s \rightarrow + \infty} n_s = +\infty
\]
and
\[
  \lim_{s \rightarrow + \infty} \frac{k_s}{n_s} = R \qquad
  \lim_{s \rightarrow + \infty} \frac{d_s}{n_s} = \delta.
\]

\begin{remark}
  Usually in the literature, the sequences ${(k_s/n_s)}_s$ and
  ${(d_s/n_s)}_s$ are not supposed to converge and $\limsup$'s
  are used instead of actual limits.
\end{remark}

In this setting, the question of the achievable pairs
$(\delta, R) \in [0,1] \times [0,1]$ remains open.
Some bounds are known:
\begin{itemize}
\item[\textbullet] Singleton bound immediately entails that $R+ \delta \leq 1$;
\item[\textbullet] A more precise bound called {\em Plotkin bound} entails
  that $R+ \delta \leq 1 - \frac 1 q$.
  See for instance \cite[Chap.~4]{couvreurLN}
\item[\textbullet] A principle that ``constructing bad codes from good ones is
  always possible'' permits to prove that give an achievable pair
  $(\delta, R)$ any pair $(\delta', R')$ with $\delta' \leq \delta$
  and $R' \leq R$ is achievable too.
  \begin{exercise}
    Prove this last assertion.
  \end{exercise}
\item[\textbullet] More precisely, it has been proved by Manin
  \cite{manin1984}, that the frontier between the subdomain of
  $[0,1] \times [0,1]$ of achievable pairs $(\delta, R)$ and the non
  achievable ones is the graph of a continuous function
  $R = \alpha_q(\delta)$. However, if proving the existence and the
  continuity of this function $\alpha_q$ is not very hard, having an
  explicit description of it remains a widely
  open problem.  An upper bound for $\alpha_q$ is given by the minimum
  of all the known upper bounds on the achievable pairs $(\delta, R)$.
\item[\textbullet] On the other hand a famous result on the average
  behaviour of random codes referred to as the Gilbert--Varshamov
  bound asserts that for a random code\footnote{This can be formalised
    as follows, consider the set of all codes of length $n$ and
    dimension $Rn$ in $\Fq^n$. This set is finite, and let $\CC$ be a
    random variable uniformly distributed over this set.}
  $\CC \subseteq \Fq^n$ with fixed rate $R$, then for any
  $\varepsilon > 0$ the probability that the relative distance
  $\delta$ of $\CC$ satisfies
  \[
    R \in [1 - H_q(\delta) - \varepsilon, 1- H_q(\delta) + \varepsilon
    ],
  \]
  goes to $1$ when $n$ goes to infinity.
  The function $H_q(\cdot)$ is the $q$--ary entropy function
  defined as
  \[
    H_q : \map{[0,1]}{\mathbb
      R}{x}{\left\{
        \begin{array}{ccc}
          -\log_q(q-1) - x \log_q(x) - (1-x)\log_q(1-x)& \text{if} &
                                                   x \neq 0,1 \\
          0 & \text{otherwise.}&
        \end{array}
      \right.
    }
  \]

  In short, the pair $(\delta, R)$ for a random sequence satisfies
  $R = 1 - H_q(\delta)$.  
\end{itemize}

In summary, the unknown function $\delta \mapsto \alpha_q(\delta)$
whose graph is the frontier of the domain of achievable pairs
$(\delta, R)$ is known to be continuous, to be bounded from below by
the Gilbert--Varshamov bound $ \delta \mapsto 1 - H_q(\delta)$ and
bounded from above by the min of all known upper bounds.  For a long
time, it has been supposed that Gilbert Varshamov bound was optimal
and that somehow, no family of codes could asymptotically beat random
codes. A breakthrough is due to Tsfasman, Vl\u{a}du\c{t} and Zink
\cite{tsfasman1982mn} who showed that the asymptotic Gilbert
Varshamov bound is not always optimal.  More precisely, they proved
the following statement.

\begin{theorem}
  Let $q = p^2$ where $p$ is a prime number. Then for any
  $R \in [0, 1]$, there exists a sequence of codes whose length goes to
  infinity and whose asymptotic parameters $(\delta, R)$ satisfy
  \[
    R+\delta \geq 1 - \frac{1}{p-1}\cdot
  \]
\end{theorem}

\begin{remark}
  Actually, the result holds for any $q = p^{2m}$ where $p$ is prime
  and $m\geq 1$.
\end{remark}

\begin{remark}
  Actually, the result on codes is the corollary of a statement on the
  existence of a sequence of algebraic curves with specific properties
  (see further Theorem~\ref{thm:main}). This statement on curves has
  proved by Tsfasman, Vl\u{a}du\c{t} and Zink in \cite{tsfasman1982mn}
  and independently by Ihara in \cite{ihara1981jfsuTokyo}.  However,
  Ihara did not rely this result with coding theory while Tsfasman {\em
    et. al.} did.
\end{remark}

\begin{figure}[!h]
  \centering
  \includegraphics[scale = 1]{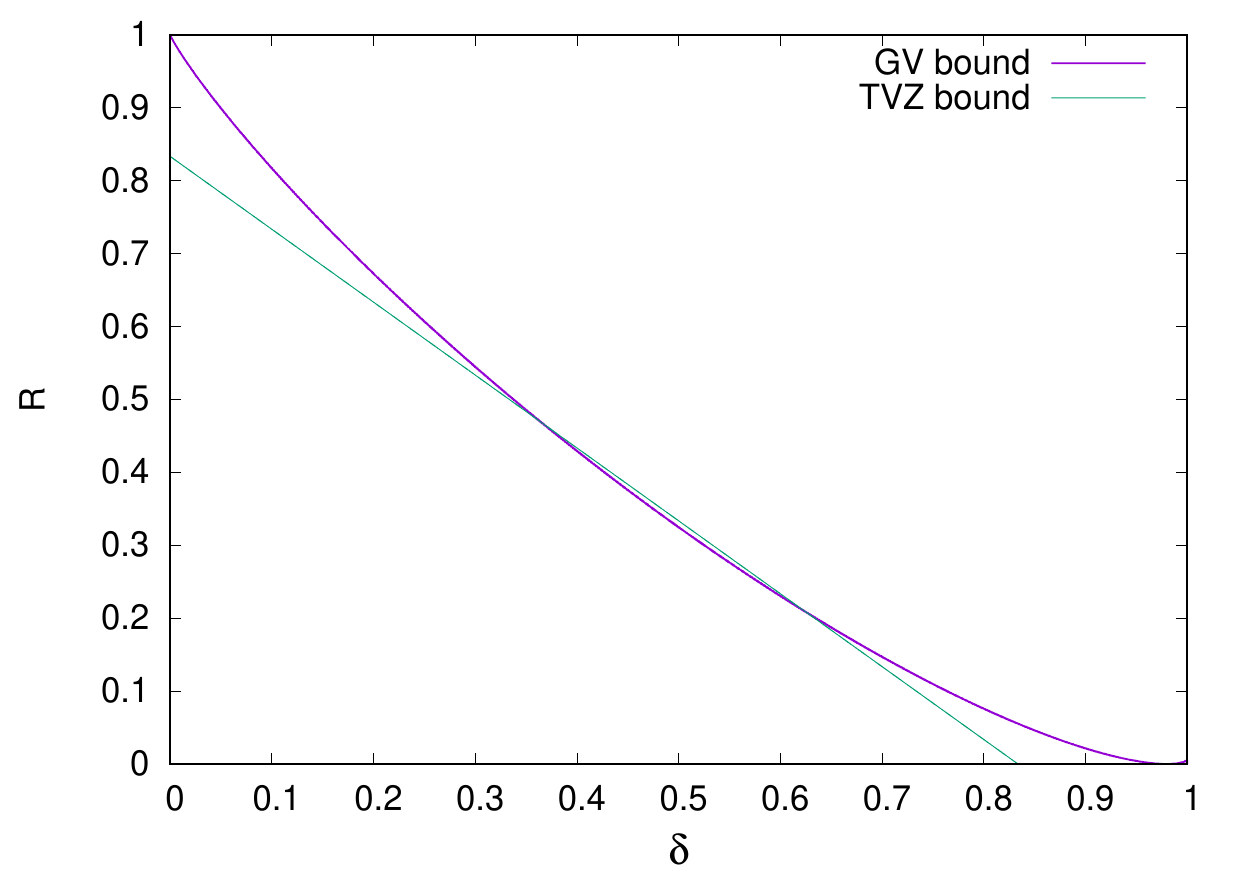}
  \caption{The TVZ bound for $q = 49$}
  \label{fig:TVZ7}
\end{figure}

\begin{figure}[!h]
  \centering
  \includegraphics[scale = 1]{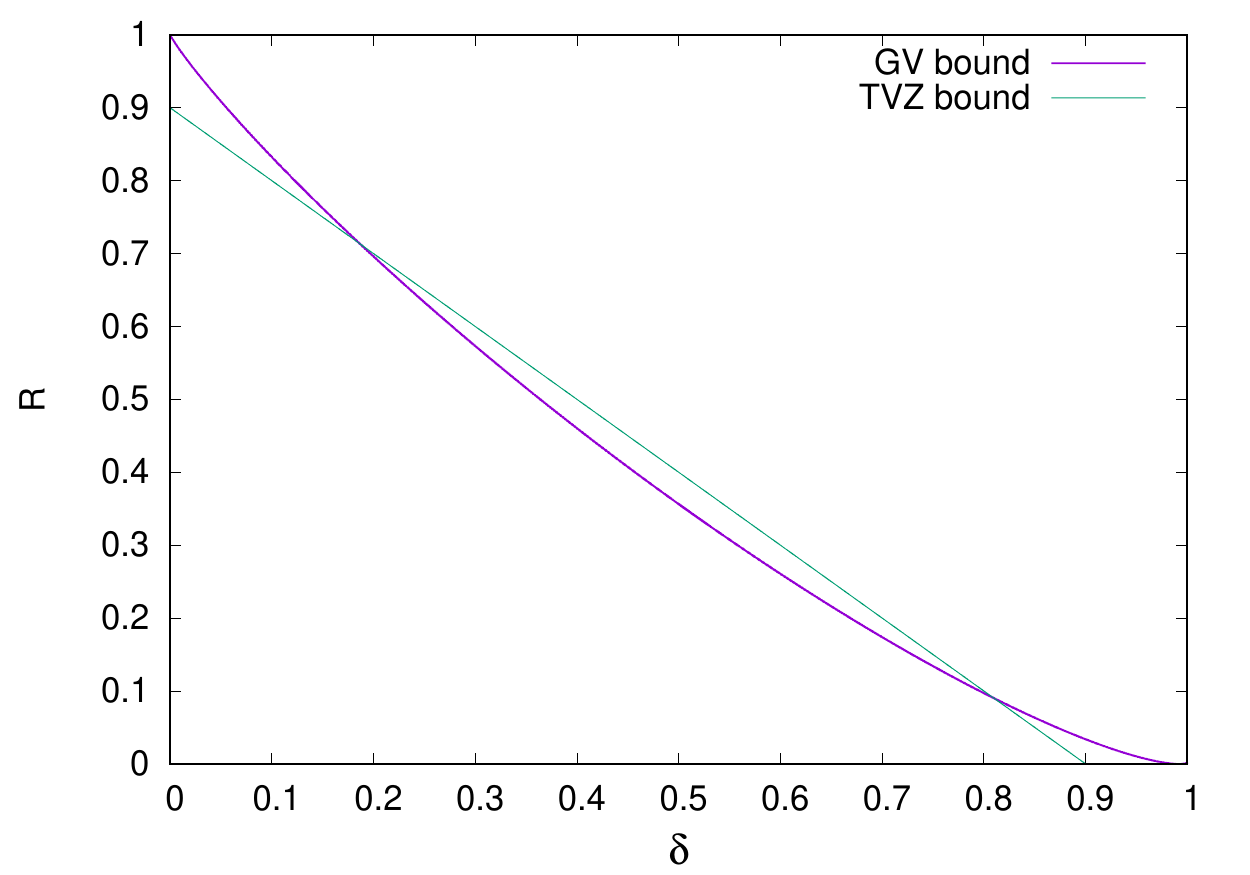}
  \caption{The TVZ bound for $q = 121$}
  \label{fig:TVZ11}
\end{figure}

It turns out that, as illustrated by Figures~\ref{fig:TVZ7}
and~\ref{fig:TVZ11}, for $q \geq 49$, such codes beat Gilbert
Varshamov bound. These codes, are actually far from being random and
are constructed using elegant techniques from number theory and
algebraic geometry.  The objective of these notes is to outline a
proof of this incredible result, which is probably one of the major
breakthroughs of coding theory.

 \section{Algebraic curves}\label{sec:curves}
The objective of this section is {\bf not} to provide an in depth
lecture of algebraic geometry but only to give the minimal
prerequisites in algebraic geometry to understand the sequel of these
notes.  In particular, here most of the proofs will be omitted. I
encourage any reader who feels comfortable with algebraic geometry and
for whom reading Harsthorne's book \cite{hartshorne1977book} is not
harder than reading Harry Potter to {\bf skip} this section for two
reasons:
\begin{itemize}
\item[\textbullet] she/he will not learn anything in it;
\item[\textbullet] for a reader who feels comfortable with the
  language of schemes, the contents of this section could
  appear to be dirty.
\end{itemize}
If you wish further details on algebraic geometry, I can encourage
the following readings depending from your knowledge on the topic:
\begin{itemize}
\item[\textbullet] Walker's book \cite{walker2000book} is an excellent first
  reading if you do not know anything about algebraic geometry and
  algebraic geometry codes.
\item[\textbullet] If you do {\bf not} like geometry, Stichtenoth's
  book \cite{stichtenoth2009book} proposes an excellent introduction
  to algebraic geometry codes from a purely arithmetic point of
  view. It provides in particular a different proof of the
  Tsfasman--Vl\u{a}du\c{t}--Zink theorem  based on so--called
  {\em recursive towers of function fields} which excludes
  any geometric consideration.
\item[\textbullet] A more advanced presentation on algebraic geometry
  codes appears in Tsfasman Vl\u{a}du\c{t} and Nogin's book
  \cite{tsfasman2007book} and Stepanov \cite{stepanov1999} .
\item[\textbullet] Finally, the reader interested in discovering
  algebraic geometry out of the context of algebraic coding theory is
  encouraged to look (for instance) at the books
  \cite{fulton1989book,shafarevich1994book}.  Lorenzini's book
  \cite{lorenzini1996book} can be an excellent reading either if you
  wish a better focus on the arithmetic side.
\end{itemize}

\subsection{Plane curves and functions}
Let $\K$ be a perfect field\footnote{In the sequel the fields of
  interest will be either $\C$ or finite fields $\Fq$.} and $\Kbar$
be its algebraic closure.  We denote by $\A^2(\Kbar)$ and
$\P^2(\Kbar)$ respectively the affine and projective planes over
$\Kbar$.  An affine {\em plane curve} $\X$ over $\K$ is the vanishing
locus in $\A^2(\Kbar)$ of a nonzero two variables polynomial
$f(x,y) \in \K [x,y]$. Similarly, a {\em projective plane curve} is
the vanishing locus in $\P^2(\Kbar)$ of a nonzero homogeneous
polynomial $F(X,Y,Z) \in \K[X,Y,Z]$. Recall that the projective plane
$\P^2(\Kbar)$ is the set of vectorial lines of $\Kbar^3$ or
equivalently is the quotient set
\[
  \P^2(\Kbar) \eqdef (\Kbar^3 \setminus \{0\})/\Kbar^\times, 
\]
and its elements are represented as triples $(u:v:w)$ with the equivalence
relation $(u:v:w) \sim (a:b:c)$ if there exists $\lambda \in \Kbar^\times$
such that $u = \lambda a$, $v = \lambda b$ and $w = \lambda c$.

Such an affine (resp. projective) curve is said to be {\em
  irreducible} if $f$ (resp. $F$) is an irreducible polynomial in
$\K [x,y]$ (resp. $\K [X,Y,Z]$) and {\em absolutely irreducible} if
$f$ (resp. $F$) is irreducible when regarded as an element of
$\Kbar[x,y]$ (resp. $\Kbar[X,Y,Z]$).

\begin{example}
  Suppose $\K = \Q$ and consider the affine curve $\X$ with equation
  $x^2 - 2y^2=0$. This curve is irreducible but {\bf not} absolutely
  irreducible. Indeed, over $\overline{\Q}$, the equation of the curve
  factorizes as $(x - \sqrt{2}y)(x + \sqrt{2}y) = 0$ and this
  factorisation is not defined over $\Q$: the polynomial $x^2 - 2 y^2$
  is irreducible over $\Q$ but not over $\overline{\Q}$.
  Geometrically speaking, $\X$ is the union of the two lines with
  respective equations $x - \sqrt{2} y = 0$ and $x + \sqrt{2} y = 0$.
  These lines are not defined over $\Q$ but their union is.
\end{example}

Given an affine irreducible plane
curve $\X$, the quotient ring $\K [x, y]/(f)$ is integral and its
field of fractions $\text{Frac}(K[x,y]/(f))$ is well--defined and
referred to as the {\em function field} of $\X$. In the projective
setting, the function field can also be defined as the field of
fractions $\frac{A(X,Y,Z)}{B(X,Y,Z)}$ where $A,B$ are homogeneous
polynomials of {\bf the same degree} with $B$ is not divisible by $F$
and with the relation:
\[
  \frac{A(X,Y,Z)}{B(X,Y,Z)} = \frac{C(X,Y,Z)}{D(X,Y,Z)} \quad
  \text{if} \quad F \ \text{divides} \ (AD - BC).
\]

For an affine curve $\X$, elements of $\K[x,y]/(f)$ can be understood
as restrictions of polynomial functions to the curve $\X$. Indeed,
considering two polynomials $a(x,y), b(x,y) \in \K [x,y]$ regarded as
functions $\A^2(\Kbar) \rightarrow \Kbar$, one can consider their
restrictions to $\X$ and a well--known result usually called {\em
  Hilbert's Nullstellensatz} (see for instance
\cite[\S~1.7]{fulton1989book}) asserts that their restrictions to $\X$
are the same if and only if $f$ divides $a-b$ and hence if and only if
they are congruent modulo the ideal spanned by $f$.

In the projective setting, a homogeneous polynomial cannot be
interpreted as a function $\P^2(\Kbar) \rightarrow \Kbar$ since an
element of $\P^2(\Kbar)$ is described by a triple $(u:v:w)$ but also
by any other triple $(\lambda u : \lambda v : \lambda w)$ for any
$\lambda \in \Kbar^{\times}$. Hence, given a non constant homogeneous
polynomial $P \in \K[X,Y,Z]$ of degree $d > 0$, the evaluation cannot
make sense since
$P(\lambda u, \lambda v, \lambda w) = \lambda^d P (u,v,w)$.  Note
however that, for such a polynomial, vanishing at a point is a
well--defined notion. Moreover, the evaluation of a fraction
$P/Q$ of two homogeneous polynomials with the same degree makes sense
since
\[
  \frac{P(\lambda u, \lambda v, \lambda w)}{Q(\lambda u, \lambda v, \lambda w)}
  = \frac{\lambda^d P(u,v,w)}{\lambda^d Q(u,v,w)} =
  \frac{P(u,v,w)}{Q(u,v,w)}\cdot
\]
This is the reason why we introduce these objects as the good definition
of functions on a projective curve.

\begin{remark}\label{remon_rationality}
  Note that we are juggling with $\K$ and $\Kbar$.  Here it is crucial
  no notice that the curve is defined as a set of points with
  coordinates in $\Kbar$, while functions, should be rational
  functions with coefficients in $\K$. On one hand, the function field
  is defined over $\K$ and describes the arithmetic of the curve. On
  the other hand, when describing a curve as a set of points,
  considering only the points with coordinates in $\K$ would be too
  poor: think for instance about the case where $\K$ is a finite
  field, in this situation the set of points with coordinates in $\K$
  is finite and might actually be empty! Then, very different
  equations may provide the same set of points with coordinates in
  $\K$ while the sets of points over $\Kbar$ will be very
  different. This explains the rationale behind considering the points
  with coordinates in $\Kbar$.
\end{remark}

\begin{remark}
  Note that when speaking about {\em functions}, these objects may
  not be defined everywhere on the curve and may have some poles somewhere.
  These objects can be understood as the algebraic geometric counterpart
  of meromorphic functions in complex analysis.
\end{remark}

\begin{remark}\label{rem:same_function_field}
  It is well--known that the projective plane can be covered by affine
  planes sometimes called {\em affine charts}. Indeed one can embed
  the affine plane into $\P^2$ as:
  \[
    \map{\A^2(\Kbar)}{\P^2(\Kbar)}{(x,y)}{(x:y:1)}
    \quad \text{or}\quad
    \map{\A^2(\Kbar)}{\P^2(\Kbar)}{(x,y)}{(x:1:y)}
    \quad \text{or}\quad
    \map{\A^2(\Kbar)}{\P^2(\Kbar)}{(x,y)}{(1:x:y).}
  \]
  The images of these three embeddings cover the full projective plane.
  Hence, given a projective curve, one can consider the restriction of
  the curve on the image of one of the above embeddings and get an
  affine curve. Practically, starting with a projective curve with
  equation $F(X,Y,Z) = 0$ one can consider for instance the affine
  curve with equation $F(x,y,1)=0$ but also those with equations
  $F(x,1,y) = 0$ or $F(1,x,y) = 0$. Hence, one can deduce affine
  curves (affine charts) from a given projective curve.  On the other
  hand, starting from an affine curve $\X$ with equation $f(x,y) = 0$
  the homogeneous polynomial $F(X,Y,Z)$ of degree $\deg f$ such that
  $f(x,y) = F(x,y,1)$ (such a {\em homogeneization} is unique, details
  are left to the reader) is the equation of a
  curve sometimes referred to as the {\em projective closure} of $\X$.

  A crucial fact is that a curve and its projective closure share a common
  object : {\bf their function field remains the very same one}.
\end{remark}

\subsection{Points}
A {\em point} of $\X$ is an element $(a,b)\in \A^2 (\Kbar)$
(resp. $(u:v:w) \in \P^2 (\Kbar)$) such that $f(a,b) = 0$
(resp. $F(u,v,w) = 0$). A point is said to be {\em a rational point}
or a $\K$--point if its coordinates all lie in $\K$.  More generally,
given an extension $\mathbb L / \K$, one can define the notions of
$\mathbb L$--points of $\X$. The set of $\K$--points or
$\mathbb L$--points of $\X$ respectively denoted by $\X (\K)$ and
$\X (\mathbb L)$. One of topics of interest for us in the sequel is
the case $\K = \Fq$. In this situation, one sees easily that $\X(\Fq)$
is finite. Indeed, it is a subset of $\A^2 (\Fq)$ or $\P^2(\Fq)$ which
are both finite sets. On the other hand $\X$ has been defined as a set
of $\Kbar$--points that we sometimes call the {\em geometric points}
in the sequel, hence we can also denote it as $\X (\Kbar)$ when
we wish to emphasize that we are interested in {\em any} possible
point.

Given an affine (resp. projective) curve $\X$ defined by the equation
$f(x,y)=0$ (resp. $F(X,Y,Z)=0$) over a field $\K$, a point
$P \in \X (\Kbar)$ with coordinates $(x_P, y_P)$ (resp.
$(u_P : v_P : w_P)$) is said to be {\em singular} if
\[
  \frac{\partial f}{\partial x}(x_P, y_P) = \frac{\partial f}{\partial
    y}(x_P, y_P) = 0 \]
resp.
\[
  \frac{\partial F}{\partial X}(u_P, v_P,w_P) =
  \frac{\partial F}{\partial Y}(u_P, v_P,w_P) =
  \frac{\partial F}{\partial Z}(u_P, v_P,w_P) = 0. 
\]
A non singular point is said to be {\em regular}. A curve without
singular points is said to be {\em regular} or {\em smooth}.
On the other hand a curve having at least one singular point
is said to be {\em singular}. It can be proved that the set
of singular points of a curve is always finite.

\bigskip

 From now on, unless otherwise specified, {\bf any {\em
  curve} is smooth projective and absolutely irreducible.}

\bigskip

\subsection{Galois action on points}
Recall that, for the sake of simplicity, we restrict the
definitions to the case where the base field $\K$ is perfect. This is
not a strong restriction for the subsequent purpose where $\K$ will
always be either finite or of characteristic zero.

Given a curve $\X$ defined over $\K$, any point $P \in \X (\Kbar)$ has
coordinates $(x_P, y_P)$ (or $(u_P:v_P:w_P)$ in the projective setting).
These coordinates being in $\Kbar$ while $\X$ is defined by polynomial
equations with coefficients in $\K$, there is a natural action of
$\text{Gal}(\Kbar / \K)$ on points of $\X$.  Note that the coordinates
of $P$ are algebraic over $\K$ and hence generate a finite extension
of $\K$ usually denoted $\K(P)$. Therefore, even if
$\text{Gal} (\Kbar / \K)$ may be a complicated object (a profinite
group), $P$ is stabilized by $\text{Gal} (\Kbar /\K (P))$ and hence
the orbit of $P$ is a finite set which is nothing but the orbit of $P$
under the action of the finite group $\text{Gal}(\K (P)' / \K)$, where
$\K(P)'$ is the Galois closure of $\K(P)$ over $\K$.

\begin{definition}\label{def:closed_point}
Let $\K$ be a perfect field, a {\em closed point} of a curve $\X$
defined over $\K$ is the orbit of a geometric point
$P \in \X (\Kbar)$ under the action of the absolute Galois group
$\text{Gal}(\Kbar / \K)$.  
\end{definition}

The number of elements in such an orbit is referred to as the {\em
  degree} of the closed point. It is also the extension degree
$[\K (P) : \K]$.  A rational point is always closed
since it is fixed by any element of the absolute Galois group and
hence it equals to its own orbit under this group action.

\begin{remark}
  If you prefer the language of number theory, closed points are
  nothing but the geometric analogue of the {\em places} of the
  function field $\K (\X)$.  
\end{remark}

\begin{example}\label{ex:circle}
  Consider the case $\K = \Q$ and the affine curve $\curve C$ with equation
  $x^2 + y^2 - 1 = 0$ (a circle). The point with coordinates $(1,0)$ is a
  rational point of $\X$, \ie{} an element of $\curve{C}(\Q)$. The complex
  point $(2, \sqrt{3} i)$, (where $i^2 =- 1$), is a geometric point
  of $\curve C$, \ie{} an element of $\curve C (\C)$.
  Finally, $\{(2, i \sqrt{3}), (2, - i \sqrt{3})\}$ is a closed point
  of degree $2$ of $\curve C$.
\end{example}

\subsection{Maps between curves}\label{ss:rat_maps}
As usually in algebra, once structures have been introduced:
for instance groups, rings, modules, etc., one introduces morphisms
between these objects. In the case of curves, we are interested in two
kinds of maps referred to as {\em morphisms} and {\em rational maps}.
A {\em rational map} between two affine (resp. projective) curves
$\X, \Y$ contained in $\A^2$ (resp. $\P^2$) is a map:
\[
  \varphi : \ratmap{\X}{\Y}{(x,y)}{(\varphi_1(x,y), \varphi_2(x,y))}
\]
resp.
\[
  \psi : \ratmap{\X}{\Y}{(u:v:w)}{(\psi_1(u,v,w), \psi_2(u,v,w),
    \psi_3(u,v,w)).}
\]
where $\phi_1, \phi_2$ (resp. $\psi_1, \psi_2, \psi_3$) are elements
of $\K(\X)$ (and, in the projective setting, at least one of the three
functions $\psi_1, \psi_2, \psi_3$ is nonzero).  The dashed arrow
$\dashrightarrow$ is here to emphasize the fact that this map is not
defined at every point but only on a subset\footnote{This subset turns
  out to be dense for a suitable topology called {\em Zariski
    topology}}.  For affine curves, at any point where
$\varphi_1, \varphi_2$ have no pole, the map is defined and said to be
{\em regular}.  For projective curves, at any point $P$ where for some
$\eta \in \K(\X)^\times$, $\eta\psi_1, \eta\psi_2, \eta\psi_3$ have no
pole at $P$ and are not simultaneously vanishing, the map $\psi$ is
well--defined and said to be {\em regular} at $P$.  A rational map
between two curves $\X \dashrightarrow \Y$ is said to be {\em regular}
if it is regular at any point of $\X$.

A rational map $\varphi : \X \dashrightarrow \Y$ induces a field extension
the other way around $\K (\Y) \hookrightarrow \K (\X)$ which is
defined as follows:
\[
  h \in \K (\Y) \longmapsto h \circ \varphi \in \K (\X).
\]
The {\em degree} of $\varphi$ is defined as the degree of this field
extension.

\begin{example}\label{ex:map_cercle}
  Back to example~\ref{ex:circle}.
  The map
  \begin{equation}\label{eq:proj_cercle}
    \ratmap{\curve C}{\P^1}{(x,y)}{(x:1)}
  \end{equation}
  is a rational map.
  It is also possible to construct a rational map $\P^1 \rightarrow \curve C$
  as
  \begin{equation}\label{eq:rat_param}
    \ratmap{\P^1}{\curve C}{(u:v)}{\left( \frac{v^2 - u^2}{u^2+v^2},
        \frac{2uv}{u^2+v^2} \right).}
  \end{equation}
  Note that these two maps are not inverses to each other.
\end{example}

Finally, the following statements are well--known. Their proofs are omitted. 

\begin{proposition}\label{prop:on_rational_maps}
  Let $h : \X \rightarrow \Y$ be a rational map between two smooth
  projective absolutely irreducible curves $\X, \Y$.
  \begin{enumerate}[(i)]
  \item if $\X$ is smooth, then $h$ is regular;
  \item if $h$ is non constant, then it is surjective.
  \end{enumerate}
\end{proposition}

\subsection{Valuations}
Recall that a {\em local ring} is a ring having a unique maximal
ideal. The term {\em local} comes precisely from the fact that many
such rings may be understood as rings of functions characterized by a
local property. For instance, given an affine curve $\X$ and a
rational point $P$ with coordinates $(x_P,y_P)$, the ring $\O_{\X, P}$
defined as the subring of $\K(\X)$ of functions which are regular
(\ie{} have no pole) at $P$. Namely
\[
  \O_{\X, P} \eqdef \left\{ \frac{a(x,y)}{b(x,y)} \in \K(\X) ~\Big|~
    b(x_P, y_P) \neq 0 \right\}.
\]
One can prove that this
ring is a local one whose maximal ideal is the ideal:
\[
  \mathfrak{m}_{\X, P} \eqdef \left\{ \frac{a(x,y)}{b(x,y)} \in \K(\X)
    ~\Big|~ b(x_P, y_P) \neq 0 \ \text{and} \ a(x_P, y_P) = 0
  \right\}.
\]
When the point $P$ is smooth, the ring $\O_{\X, P}$ is known to be a
discrete valuation ring, which means that the maximal ideal
$\mathfrak m_{\X, P}$ is principal and that, given a generator $t$ of
this maximal ideal, for any nonzero element $a \in \O_{\X, P}$, there
exists a non negative integer $n$ and an element
$\varphi \in \O^\times_{\X, P}$ such that $a = \varphi t^n$.  Such a
generator $t$ of $\mathfrak m_{\X, P}$ is called a {\em local
  parameter} (or sometimes a {\em uniformising parameter}) at $P$.
Moreover, the integer $n$ does not depend on the choice of the
generator $t$ and is referred to as the {\em valuation} of $a$ at $P$
and denoted as $v_P (a)$.  Next, one can easily prove that $\K (\X)$
is nothing but the field of fractions of $\O_{\X, P}$. Then, any
function $h \in \K(\X)$ can be written as
$h = \frac{h_1}{h_2} \in \K (\X) \setminus \{0\}$, where
$h_1, h_2 \in \O_{\X, P}$ and the valuation of $h$ at $P$ will be
defined as
\[
  v_P(h) = v_P (h_1) - v_P(h_2).
\]
In summary, we introduced a map
\[
  v_P : \K(\X) \setminus \{0\} \rightarrow \Z
\]
and this map is known to satisfy the following properties,
\begin{enumerate}[(i)]
\item[\textbullet] $\forall a,b \in \K(\X)\setminus \{0\}$, $v_P(ab) = v_P(a)+v_P(b)$;
\item[\textbullet] $\forall a,b \in \K(\X)\setminus \{0\}$,
  $v_P(a+b) \geq \min \{v_P(a),v_P(b)\}$ and equality holds when
  $v_P(a) \neq v_P(b)$.
\end{enumerate}

Finally, it should be emphasized that, even if we defined the notion
at a rational point, one can actually extend the notion to any
geometric point by replacing $\K (\X)$ by $\Kbar (\X)$, \ie{} the
field of rational functions on $\X$ with coefficients in $\Kbar$.
Therefore, the valuation may be defined at any possible point.

\subsection{Divisors}
A fundamental object when studying the geometry and arithmetic of a
curve is divisors which somehow are the curve/function fields
counterpart of fractional ideals in the theory of number fields.

Given a {\bf smooth} curve $\X$ over a perfect field $\K$, a
(geometric) {\em divisor} is a formal $\Z$--linear combination of
geometric points of $\X$. A divisor is said to be {\em rational}
if it is globally invariant under the action of $\text{Gal}(\Kbar/\K)$.
Equivalently, it is a formal sum of closed points of $\X$.

Hence a divisor $G$ on $\X$ can be
represented as
\begin{equation}\label{eq:divisor}
  G = n_1 P_1 + \cdots + n_r P_r,
\end{equation}
where the $n_i$'s are integers and the $P_i$'s are geometric points of
$\X$.  The set $\{P_1, \dots , P_r\}$ is referred to as the {\em
  support} of $G$. The divisor is rational if for any
$i, j \in \{1, \dots, r\}$ such that $P_i, P_j$ are in the same orbit under
the action of $\text{Gal}(\Kbar/\K)$, then $n_i = n_j$.

\begin{remark}
  We emphasize that a sum of rational points yields a rational
  divisor but the converse is false. A rational divisor may be a sum
  of non rational points. See the subsequent Example~\ref{ex:rat_div}.
\end{remark}

\begin{example}\label{ex:rat_div}
  Back to the curve of Example~\ref{ex:circle} defined over $\Q$ with
  equation $x^2+y^2-1 = 0$, consider the points
  $P = (\frac 1 2, \frac{\sqrt{3}}{2}), P' = (\frac 1 2,
  -\frac{\sqrt{3}}{2})$ and $Q = (1,0)$. Then, $aP+bP'+cQ$ is a
  rational divisor on the curve if and only if $a=b$.
\end{example}

The group of divisors is equipped with a partial order relation denoted
$\leq$ and defined as follows. Given two divisors
\[
  G = \sum_{P \in \X (\Kbar)} n_P P \quad \text{and} \quad G' =
  \sum_{P \in \X (\Kbar)} n'_P P,
\]
we say that $G \leq G'$ if
\[
\forall P \in \X (\Kbar),\ n_P \leq n'_P.
\]
In particular, a divisor $G$ is said to be {\em positive} if $G \geq 0$,
where $0$ denotes the zero divisor.

Given a divisor $G$ as in \eqref{eq:divisor}, its {\em degree} is
defined as
\[
  \deg G \eqdef n_1 + \dots + n_r.
\]

Given a function $f \in \K (\X) \setminus \{0\}$, one can associate
its divisor denoted $(f)$ and defined as
\begin{equation}\label{eq:princ_div}
  (f) \eqdef \sum_{P\in \K (\X)} v_P(f) P.
\end{equation}
Such a divisor is called a {\em principal divisor}.

\begin{remark}
  For such an object to be a divisor, we need to show that the
  sum~(\ref{eq:princ_div}) is finite, \ie{} that the $n_P$'s are all
  zero but a finite number of them. This is actually due to a well--known
  fact appearing in the next statement whose proof is omitted.
\end{remark}

\begin{proposition}
  A nonzero rational function on a curve has only a finite number of
  zeroes and poles.
\end{proposition}

\begin{remark}\label{rem:princ=rat}
  It is worth noting that a principal divisor is rational.  Indeed,
  one can first note that, since $f \in \K (\X)$ and hence has its
  coefficients in $\K$, then for any geometric point $P \in \X(\Kbar)$
  and any $\sigma \in \text{Gal}(\Kbar/\K)$ then
  $v_P(f) = v_{\sigma (P)}(f)$.
\end{remark}

The following very classical statement is crucial in the
sequel. 

\begin{proposition}\label{prop:degree_principal}
 The degree of principal divisor is always $0$.
\end{proposition}

We finish this discussion with a statement that we admit and which will
be useful later.

\begin{proposition}\label{prop:constant_functions}
  A principal divisor $(f)$ associated to $f \in \K (\X)^\times$ is zero if
  and only if $f$ is constant.
\end{proposition}

\subsection{Genus and Riemann--Roch Theorem}\label{ss:genus_and_RR}
The most elementary curve one may define is the affine line
$\A^1$ and its projective closure being the projective line $\P^1$.
Regular functions on $\A^1$ are nothing but univariate polynomials.
Regarding such a polynomial $h(x) \in \K [x]$ as rational function on
$\P^1$, it has a pole at the point ``at infinity'', \ie{} the
points with homogeneous coordinates $(1:0)$ and one can prove that the
valuation at this pole is nothing but $-\deg h$.

Therefore, the space $\K [x]_{\leq n}$ of polynomials of degree less
than or equal to $n$ can be (with enough pedantry) defined as the space of
rational functions on $\P^1$ which are regular everywhere on an affine
chart and with valuation larger than or equal to $-n$ at the point at
infinity. Denoting by $P_{\infty}$ this point at infinity, then the
space $\K [x]_{\leq n}$ can be regarded as the space of rational
functions $h \in \K (\P^1)$ which are either $0$ or such that
\[
  (h) \geq - n P_{\infty}.
\]
As the following definition suggests, Riemann--Roch spaces are
generalisations for curves of the spaces $\K [x]_{\leq n}$.

\begin{definition}[Riemann--Roch space]
  Let $\X$ be a smooth projective absolutely irreducible curve over $\K$
  and $G$ be a rational divisor on $X$. Then the {\em Riemann--Roch space
    associated to $G$} is defined as
  \[
    L(G) \eqdef \left\{h \in \K(\X) ~|~ (h) + G \geq 0 \right\} \cup \{0\}.
  \]
  This is a vector space over $\K$.
\end{definition}

\begin{remark}
  According to the previous discussion, on $\P^1$, we have
  $L(nP_{\infty}) \simeq \K [x]_{\leq n}$.
\end{remark}

The following statement summarises some properties of Riemann--Roch spaces.

\begin{proposition}\label{prop:RRspaces}
  \begin{enumerate}[(i)]
  \item A Riemann--Roch space is a vector space over $\K$ of finite dimension;
  \item\label{item:RRzero} For any rational divisor $G < 0$, we have
    $L(G) = \{0\}$;
  \item\label{item:dimRRspace} For any rational divisor $G$, we have
    $\dim_\K L(G) \leq \deg G + 1$.
  \end{enumerate}
\end{proposition}

With the above statement at hand, we can introduce a fundamental
invariant of a curve: its {\em genus}. There are dozens of manners to
define this object but none of them is
trivial. The one given in these notes is far from being satisfying since
it is clearly not intuitive. However, it permits to define the object
with a minimal amount of material.

\begin{definition}[Genus of a curve]\label{def:genus}
  Let $\X$ be a smooth projective absolutely irreducible curve. The
  {\em genus} of $\X$ is defined as
  \[
    g = 1 - \min_D \{\dim_\K L(D) - \deg D\},
  \]
  where $D$ ranges over all the divisors of $\X$.
\end{definition}

\begin{remark}\label{rem:genus_is_positive}
  Proposition~\ref{prop:RRspaces}~(\ref{item:dimRRspace}) asserts that
  the involved minimum exists and that the genus is nonnegative. 
\end{remark}

\begin{exercise}
  Prove the statement of Remark~\ref{rem:genus_is_positive}.
\end{exercise}

\begin{exercise}
  Using Definition~\ref{def:genus},
  prove that the genus of the projective line $\P^1$ is zero.
\end{exercise}

Note that the effective computation of the genus is not a simple task.
However, for smooth plane curves of degree $d$, there is a closed
formula (see \cite[Prop.~VIII.5]{fulton1989book}):
\[
  g = \frac{(d-1)(d-2)}{2}\cdot
\]
This permits in particular to prove that the projective line and smooth
conics have genus $0$.

\begin{remark}\label{rem:arith_genus}
  The notion of genus can actually be defined for singular curves.
  In this context, two distinct invariants respectively called
  {\em arithmetic genus} and {\em geometric genus} can be defined.
  These two invariants coincide when the curve is smooth.
\end{remark}

We conclude this section with Riemann--Roch Theorem, which is a crucial
statement in the theory of algebraic curves. This statement is admitted and
we refer the reader to Fulton \cite{fulton1989book} or Stichtenoth's
\cite{stichtenoth2009book} book for a proof. The first part of the statement
is actually a straightforward consequence of the definition we gave
for the genus (Definition~\ref{def:genus}).

\begin{theorem}[Riemann--Roch Theorem]\label{thm:RR}
  Let $\X$ be a smooth absolutely irreducible curve of genus $g$ over
  $\K$ and $G$ be a rational divisor on $\X$. Then
  \[
    \dim_{\K} L(G) \geq \deg G + 1 - g
  \]
  and equality holds when $\deg G > 2g-2$.
\end{theorem}

\subsection{The Riemann--Hurwitz formula}
The last statement that will be useful in the sequel is
Riemann--Hurwitz formula which relates the genera of two smooth
projective absolutely irreducible curves $\X, \Y$ linked by a non
constant rational map $\varphi : \X \dashrightarrow \Y$.  Recall that,
according to Proposition~\ref{prop:on_rational_maps}, such a map is
regular and surjective. Denoting by $\delta$ its degree (see
\S~\ref{ss:rat_maps} for the definition of degree), consider any
geometric point $P \in \Y(\Kbar)$. Then one can prove that
$\varphi^{-1}(\{P\})$ is a finite subset of $\X(\Kbar)$ and that for
any $P$ but finitely many of them the cardinality of
$\phi^{-1}(\{P\})$ always equals $\delta$.

The finite number of points of $\Y(\Kbar)$ where this no longer holds
are called {\em ramified points}. Given $Q \in \X(\Kbar)$,
$P = \varphi(Q)$ and $t$ a local parameter at $P$, the {\em
  ramification index of $Q$} is defined as
\[
  e_Q \eqdef v_Q(t\circ \varphi),
\]
$t \circ \varphi$ being an element of $\K(\X)$.  It can be proved that
this definition does not depend on the choice of the local parameter
$t$ at $P$.  According to the previous definition, for any point
$Q \in \X (\Kbar)$ but finitely many of them, we have $e_Q = 1$.

Here we have the material to state Riemann--Hurwitz formula.

\begin{theorem}[Riemann--Hurwitz formula (Tame version)]\label{thm:RH}
  Let $\X, \Y$ be two smooth projective absolutely irreducible curves over $\K$
  and $\varphi : \X \dashrightarrow \Y$ be a rational map.  Suppose that
  for any $Q \in \Y (\Kbar)$, the ramification index $e_Q$ is prime to
  the characteristic of $\K$. Then, the genera $g_\X, g_\Y$ of $\X, \Y$
  are related by the following formula.
  \[
    (2g_\X - 2) = \deg \varphi \cdot (2g_\Y - 2) + \sum_{Q \in \Y(\Kbar)} (e_Q - 1).
  \]
\end{theorem}

\begin{remark}
  According to the previous discussion, the terms of the sum in the
  above formula are all zero but a finite number of them.
\end{remark}

\begin{remark}
  The assumption ``ramification indexes are prime to the characteristic''
  can be discarded at the cost of replacing the term $\sum (e_Q - 1)$
  by a more complicated one. See \cite[Thm.~3.4.13]{stichtenoth2009book}.
\end{remark}

This formula is particularly useful since many curves $\X$ are
described by a morphism $\X \rightarrow \P^1$. Since $\P^1$ is known
to have genus $0$, the genus of $\X$ can be deduced from the knowledge
of the degree of this map and the ramification indexes.

\begin{example}
  Consider the map~(\ref{eq:proj_cercle}) of
  Example~\ref{ex:map_cercle} but here we regard the curve $\curve C$
  as a curve over $\C$. One sees that any point
  $P = (t:1) \in \P^1(\C)$ has 2 preimages by the map if
  $t\notin \{-1, 1\}$ and only one if $t \in \{-1,1\}$. Therefore,
  there are two ramified points both with ramification index $2$ (one
  can show that the map does not ramify at infinity). Moreover, the
  map has degree $2$.  Then, Riemann--Hurwitz formula yields
  \[
    2g_{\curve C} - 2 = 2(2g_{\P^1}- 2) + 2.
  \]
  Since $g_{\P^1} = 0$, we deduce that $g_{\curve C} = 0$ too.
\end{example}

\subsection{What about non plane curves?}
A last important fact is that some curves are not plane and may be
contained in $\P^N$ for $N > 2$. It is actually important in the
sequel since we are searching for smooth curves $\X$ over a finite
field $\Fq$ with $\card{\X (\Fq)}$ arbitrarily large.  Since
$\card{\P^2(\Fq)}$ is finite (and equal to $q^2+q+1$) such a curve
may not be embeddable in $\P^2$ and requires a larger dimensional ambient
space. So, the question is... what remains true when considering curves
in $\P^N$ with $N> 2$? and actually, how are such objects defined?

We define a projective subvariety of $\P^N$ as the common
vanishing locus of the elements of a homogeneous ideal
$I \subseteq \K [X_0, \dots, X_N]$. If this ideal is prime, then the
variety will be said to be {\em irreducible} and in this setting, the
function field of the variety can be defined in the very same manner
as in the plane case. Then, the {\em dimension} of the variety can be
defined as the transcendence degree of the function field over $\K$.
A curve will be a variety of dimension $1$. Smoothness can be defined
very similarly by requiring a non simultaneous vanishing of all the
partial derivatives with respect to the $N+1$ variables.  All the
other objects, rational maps, valuations, divisors, Riemann--Roch
spaces can be defined in the very same manner at the cost of heavier
notation. Finally all the previous statements on plane curves actually
hold for any curve.

 \section{Algebraic geometry codes}\label{sec:ag_codes}
Now, we have the necessary material to define algebraic geometry (AG) codes.
Before, let us recall the definition of Reed--Solomon codes that
AG codes generalise.

\subsection{Reed--Solomon codes}

\begin{definition}
  Let $\alpha_1, \dots, \alpha_n$ be distinct elements of $\Fq$.
  Let $0 \leq k \leq n$, the code $\mathbf{RS}_k$ is defined as
  \[
    \mathbf{RS}_k(\alpha_1, \dots, \alpha_n) \eqdef \{(p(\alpha_1),
    \dots, p(\alpha_n)) ~|~ p \in \Fq[x]_{\leq k-1}\}.
  \]
\end{definition}

It is well--known that these codes have parameters $[n, k, n-k+1]_q$
and hence reach the Singleton bound \eqref{eq:Singleton}. However, they
are constrained in the sense that the $\alpha_i$'s should be distinct
and hence the length should be bounded by $q$. Thus, even if these
codes have optimal parameters, it is hopeless to use them in order to
construct an infinite family of codes over a fixed field $\Fq$ whose
length goes to infinity. Here, curves enter the game.  Note first that
Reed--Solomon codes may be defined in a much more pedant manner as
follows. Consider the projective line $\P^1$ and let $P_1, \dots, P_n$
be the rational points of $\P^1$ with respective homogeneous
coordinates $(\alpha_1 : 1), \dots , (\alpha_n : 1)$. Then,
$\mathbf{RS}_k(\alpha_1, \dots, \alpha_n)$ may be defined as
\[
  \mathbf{RS}_k(\alpha_1, \dots, \alpha_n) = \left\{
    (h(P_1), \dots, h(P_n)) ~|~
    h \in L((k-1)P_{\infty})\right\}.
\]
This leads to a natural generalisation to algebraic
curves. The interest being the fact that a curve may have more
rational points than the projective line and hence replacing $\P^1$
by an arbitrary curve may provide the opportunity of getting codes
of length larger than $q$.

\subsection{Algebraic geometry codes}
We give a minimal introduction to algebraic geometry (AG) codes. The
reader interested in further references is encouraged to have a look
at the surveys \cite{hoholdt98chapter,duursma2008,couvreur2021} or the
books \cite{tsfasman2007book,stichtenoth2009book}. We also refer to
\cite{hoeholdt95it,beelen08} for references on the decoding of AG
codes.

  \begin{definition}\label{def:AG_code}
    Let $\X$ be a smooth absolutely irreducible curve over $\Fq$.
    Let $\cP = (P_1, \dots, P_n)$ be an ordered sequence of distinct
    rational points of $\X$. Let $G$ be a rational divisor on $\X$ whose
    support avoids the points $P_1, \dots, P_n$.
    Then, the algebraic geometry code $\CL (\X, \cP, G)$ is defined
    as
    \[
      \CLdef  \eqdef \{(f(P_1), \dots, f(P_n)) ~|~ f \in
      L(G)\}.
    \]
  \end{definition}

  Once the codes are defined, their parameters can be evaluated
  using the previously introduced material of algebraic geometry.

  \begin{theorem}\label{thm:param_AGcode}
    Let $\X$ be a smooth absolutely irreducible curve of genus $g$
    over $\Fq$. let $\cP = (P_1, \dots, P_n)$ be a tuple of rational
    points of $\X$ and $G$ be a rational divisor on $\X$ whose support avoids
    $P_1, \dots, P_n$.
    Suppose that $\deg G < n$. Then, the
    parameters $[n, k, d]_q$ of $\CLdef$ satisfy
      \begin{eqnarray}
        k & \geq & \deg G + 1 - g \quad \text{with equality when }
                   \deg G > 2g-2;\\
        d & \geq & n - \deg G. 
      \end{eqnarray}
  \end{theorem}

  \begin{proof}
    Denote by $D_{\cP}$ the divisor
    \(
      D_{\cP} \eqdef P_1 + \cdots + P_n.
    \)
    Consider the map
    \[
      ev_{\cP} : \map{L(G)}{\Fq^n}{f}{(f(P_1), \dots, f(P_n)).}
    \]
    Its image is trivially $\CLdef$.  The kernel of this map is the
    subspace of $L(G)$ of functions $f$ vanishing at
    $P_1, \dots, P_n$. This subspace is nothing but $L(G-D_{\cP})$. By
    assumption, $\deg (G - D_{\cP}) = -(n - \deg G)$, is negative and
    hence, from Proposition~\ref{prop:RRspaces}~(\ref{item:RRzero}),
    $L(G-D_{\cP}) = \ker ev_{\cP} = \{0\}$. Thus,
    $ev_{\cP}$ is injective and 
    \[
      \dim \CLdef = \dim L(G) \geq \deg G + 1 - g,
    \]
    with equality if $\deg G > 2g-2$.  Here, the last inequality
    together with the equality case are due to Riemann--Roch Theorem
    (Theorem~\ref{thm:RR}).

    For the minimum distance, let us
    introduce $h \in L(G) \setminus \{0\}$ such that $ev_{\cP}(h)$ has
    Hamming weight $d$. It means that there exist distinct points
    $P_{i_1}, \dots, P_{i_{n-d}}$ among $P_1, \dots, P_n$ at which $h$
    vanishes. Consequently,
    $h \in L(G - P_{i_1} - \cdots - P_{i_{n-d}})$ and since
    $h \neq 0$, the space
    $L(G - P_{i_1} - \cdots - P_{i_{n-d}}) \neq \{0\}$, which, from
    Proposition~\ref{prop:RRspaces}~(\ref{item:RRzero}) again, implies
    that $\deg (G - P_{i_1} - \cdots - P_{i_{n-d}}) \geq 0$ and hence
    \[
      d \geq n - \deg G.
    \]
  \end{proof}

  Let us comment this last result. It was mentioned in \S~\ref{ss:linear_codes}
  that, from Singleton bound \eqref{eq:Singleton}, any $[n,k,d]_q$ code
  satisfies
  \[
    k+d \leq n+1.
  \]
  On the other hand, Theorem~\ref{thm:param_AGcode} asserts that an $[n,k,d]_q$
  AG code $\CLdef$ satisfies
  \[
    n+1-g \leq k+d.
  \]
  In summary, AG codes are in the worst case at ``distance $g$ from
  Singleton bound''. Thus, one can expect good codes for a ``not too
  large'' genus $g$. On the other hand, the objective is to construct
  sequences of codes whose length exceeds $q$ and more generally
  construct families of codes over $\Fq$ whose length goes to
  infinity. Thus, for the length to be large, we look for curves with
  the largest possible number of rational points.

  \subsection{The problem of infinite sequence of curves with many points compared to their genus}

  We expect to get sequences of curves over
  $\Fq$ whose genus grows slowly and number of rational points grows
  quickly. However, these two objectives are somehow in opposition:
  to get many rational points, we need a large genus. A well--known
  result due to Weil asserts that for a smooth absolutely irreducible
  curve $\X$ over $\Fq$,
  \begin{equation}\label{eq:WeilBound}
    \card{\X(\Fq)} \leq q+1+2g\sqrt{q}.
  \end{equation}
  Thus, we look for a good trade off between the genus and the number
  of rational points. Now, we have the material to reformulate our
  coding theoretic problem of producing asymptotically good infinite
  sequences of codes in terms of the construction of sequences
  of algebraic curves with specific features. For this, let us
  consider a sequence of curves ${(\X_s)}_{s \in \N}$ with sequence of
  genera ${(g_s)}_{s \in \N}$. We suppose that the sequence
  ${(\card{\X_s (\Fq)})}_{s \in \N}$ goes to infinity, hence, according
  to Weil's bound \eqref{eq:WeilBound}, the sequence of genera
  should also go to infinity.
  Let
  \begin{equation}\label{eq:gamma}
    \gamma = \limsup_{s \rightarrow + \infty} \frac{\card{\X_s (\Fq)}}{g_s}\cdot
  \end{equation}
  For any such curve in the sequence, we fix a rational divisor $G_s$
  and the sequence of rational points $\cP_s = (P_1, \dots, P_{n_s})$
  will be chosen as the full list of rational points, \ie{}
  $n_s = \card{\X_s (\Fq)}$.

  \begin{remark}
    One could ask whether it is possible to have a rational divisor $G_s$ of
    any degree whose support avoids $P_1, \dots, P_{n_s}$ while
    $\{P_1, \dots, P_{n_s}\} = \X (\Fq)$? The answer is positive, such
    divisors $G_s$ exist and the constraint that the support of $G_s$
    should avoid $\X (\Fq)$ is actually easy to satisfy. See
    \cite[Rem.~15.3.8]{couvreur2021} for a detailed discussion on this
    specific question.
  \end{remark}

  Then, the codes $\CL (\X_s, \cP_s, G_s)$ have parameters $[n_s, k_s, d_s]_q$
  satisfying
  \[
    \begin{array}{ccl}
      n_s& = & \card{\X_s (\Fq)} \\
      k_s& \geq & \deg G_s + 1 - g_s \\
      d_s& \geq & n_s - \deg G_s.
    \end{array}
  \]
  Therefore, one can eliminate $\deg G_s$ and get
  \begin{equation}\label{eq:asymp_paramAG}
    k_s + d_s \geq n_s + 1 - g_s.
  \end{equation}
  Set
  \[
    R = \limsup_{s \rightarrow + \infty} \frac{k_s}{n_s}\quad
    \text{and} \quad \delta = \limsup_{s \rightarrow + \infty}
    \frac{d_s}{n_s}\cdot
  \]
  Then, dividing \eqref{eq:asymp_paramAG} by $n_s$ and letting
  $s$ go to infinity, we get
  \[
    R+\delta \geq 1 - \frac{1}{\gamma},
  \]
  where $\gamma$ is defined in \eqref{eq:gamma}.
  Therefore, any pair $(\delta, R)$ lying on the line of equation $R + \delta =
  1 - \frac 1 {\gamma}$ is achievable.

  \begin{remark}
    Even if the term $\deg G_s$ has been eliminated, this term is
    worth in order to chose the point in the line of equation
    $R+\delta = 1 - \frac 1 \gamma$ you want to target.
  \end{remark}

  \begin{exercise}
    Prove that by choosing a relevant sequence of rational divisors $(G_s)$ on
    the curves $\X_s$, one can reach any point of the line of equation
    $R+\delta = 1 - \frac{1}{\gamma}\cdot$
  \end{exercise}

\subsection{The Ihara constant $A(q)$}
Now, we would like to estimate the optimal asymptotic parameters
$(\delta, R)$ that can be achieved. For that, let us introduce the {\em
  Ihara constant}:
  \[
    A(q) \eqdef \limsup_{g \rightarrow + \infty} \max_{\stackrel{\X,\ 
     \text{curve}}{ \text{of genus }g}} \frac{\card{\X (\Fq)}}{g}\cdot
\]
Then, the Tsfasman--Vl\u{a}du\c{t}--Zink (TVZ) bound asserts the existence of
families of codes whose asymptotic parameters $(\delta, R)$ satisfy
\[
  R+\delta \geq 1 - \frac{1}{A(q)}\cdot 
\]
This opens the question of the value of $A(q)$.
The remainder of these notes consists in outlining a proof of the
following statement.

\begin{theorem}\label{thm:main}
  For $q = p^2$ and $p$ a prime number, we have
  \[
    A(q) \geq \sqrt{q} - 1.
  \]
\end{theorem}
Combining the previous result with the TVZ bound, one ca prove that for
$p\geq 7$, and hence when $q$ is the square of a prime and is larger
than or equal to $49$, the TVZ bound exceeds the Gilbert Varshamov one,
proving that some families of codes from algebraic curves
are better than random codes.

Let us conclude with some comments.
\begin{itemize}
\item[\textbullet] Actually, the result extends to $q = p^{2m}$ for any $m \geq 1$ but the
  proof gets more complicated and involves other families of curves. Namely,
  the proof to follow involves modular curves, while the general case involves
  Shimura curves. See \cite{ihara1981jfsuTokyo,tsfasman1982mn}.
\item[\textbullet] The TVZ bound is actually optimal. Indeed,
  subsequently to the publication of Tsfasman--Vl\u{a}du\c{t}--Zink
  result, in \cite{vladut1983faa} Drinfeld and Vl\u{a}du\c{t} proved
  that for any prime power $q$, we always have $A(q) \leq \sqrt{q}+1$.
\item[\textbullet] Another proof of Theorem~\ref{thm:main}
  using a very different approach has been given by Garcia and
  Stichtenoth in \cite{GS1995}.
\end{itemize}

The core of the proof of this wonderful result rests on the use of
families of curves called {\em modular curves} which parameterise
families of algebraic curves called {\em elliptic curves}.

 \section{Elliptic curves}\label{sec:elliptic}
Elliptic curves is another fascinating topic in number theory.  They
are also a fundamental object in cryptography but this is not the
point of these notes.  In this section, we start by presenting basic
notions about these objects over an arbitrary field. Our objective is
in particular to construct these so--called {\em modular curves} which
will yield excellent codes. These modular curves are
algebraic curves which parameterise families of elliptic curves with a
specific extra structure called {\em level}.

Afterwards, in Section~\ref{sec:modular}, we will discuss elliptic
curves and modular curves over $\C$. This choice of discussing complex
curves in such notes might seem surprising while our interest will
clearly be curves over finite fields. However, a preliminary study of
the complex case presents several advantages. First, it provides a
much more intuitive presentation of the topics with the benefits of
the possible use of analytical tools. Second, even if the analytic
proofs cannot transpose in the finite field setting, they permit to
compute algebraic formulas, \ie{} polynomial equations defining
modular curves. These equations turn out to be defined over $\Z$ and
then --- and this is very far from being trivial --- their reduction
modulo $p$ will give the equation of a curve parameterising elliptic
curves over $\F_p$ or ${\overline{\F}}_p$ with some level structure.

\medskip

\noindent {\bf Note.} In this section, we assume the ground field $\K$
to have characteristic different from $2$ and $3$. Most of the material
of the present section and the subsequent one are taken from the book
\cite{silverman2009book} and the lecture notes \cite{milneMF}.

\subsection{Basic definitions}
An {\em elliptic curve} $\E$ over a field $\K$ is a smooth projective
curve of genus 1 with at least one rational point denoted by
$O_{\E}$. From Proposition~\ref{prop:constant_functions}, the
Riemann--Roch space $L(0)$ associated to the zero divisor contains
only the constant functions and hence has dimension $1$.  Then, by
Riemann--Roch Theorem, the spaces $L(2O_{\E})$ and $L(3O_{\E})$ have
respective dimensions $2$ and $3$ (note that as soon as the divisor's
degrees are positive, they are $> 2g-2$ and hence we fit in the
equality case of Riemann--Roch Theorem). Denote by $x, y$ two
functions such that
\[
  L(2O_{\E}) = \span_\K \{1, x\} \quad L(3O_{\E}) = \span_\K \{1, x, y\}.
\]
Note that these choices for $x$ and $y$ are not canonical and hence
any of the following changes of variables are admissible
\begin{equation}\label{eq:cv}
  x' \leftarrow ax+b, \ \text{with}\ a \neq 0 \qquad
  y' \leftarrow uy + vx + w, \ \text{with}\ u \neq 0.
\end{equation}
Now, consider the space $L(6O_{\E})$. It contains the functions
\[
  1, x, y, x^2, xy, x^3, y^2.
\]
Moreover, again from Riemann--Roch Theorem, $L(6O_{\E})$ has dimension $6$
and hence there is a nontrivial linear relation on these functions
\[
  y^2 + u xy + v y= ax^3 + bx^2 + cx + d. 
\]

\begin{exercise}\label{ex:anonzero}
  Prove that $y^2$ and $x^3$ should be involved in this linear relation,
  which explains why, after a renormalisation, one can suppose the coefficient
  of $y^2$ to be $1$. Deduce from this that $a \neq 0$.
\end{exercise}

Now, we perform successive changes of variables which are admissible,
\ie{} changes of variables of the form (\ref{eq:cv}). A first one\footnote{
  In order to keep light notation, we remove the ' in $x', y'$ and hence
  write the outputs of the change of variables as the input, hence the notation
  $y \leftarrow y + \frac{u}{2}x$. This is not a completely rigorous
  notation and the reader bothered by this is encouraged to rewrite this page
  by replacing the $x$'s and $y$'s as $x', x'', x'''$ and $y', y'', y'''$
at the good spots.}:
$y \leftarrow y + \frac{u}{2}x$ leads to an equation:
\begin{equation}\label{eq:eq1}
  y^2 + v_1 y = a_1 x^3 + b_1 x^2 + c_1 x + d_1,
\end{equation}
for some $a_1, b_1, c_1, d_1 \in \K $.
A change $y \leftarrow y + \frac{v_1}{2}$ yields
\begin{equation}\label{eq:eq2}
  y^2  = a_2 x^3 + b_2 x^2 + c_2 x + d_2.
\end{equation}
for some $a_2, b_2, c_2, d_2 \in \K $.
Next, a change of the form $x \leftarrow x + \frac{b_2}{3a_2}$ yields
to an equation:
\begin{equation}\label{eq:eq3}
  y^2 = a_3 x^3 + c_3x + d_3,
\end{equation}
for some $a_3, c_3, d_3 \in \K$. Finally, applying the change of variables
$x \leftarrow a_3 x,\ y \leftarrow a_3^2 y$ and dividing both sides by
$a_3^4$, we get an equation of the form
\begin{equation}\label{eq:Weierstrass}
  y^2 =  x^3 + Ax + B,
\end{equation}
for some $A, B \in \K$. Such an equation is called a {\em Weierstrass
equation} of the curve.

\begin{exercise}
  Using Exercise \ref{ex:anonzero}, check that the last
  change of variables was admissible, \ie{} that $a_3 \neq 0$.
\end{exercise}

{\bf In summary}, starting from an elliptic curve $\E$ over $\K$, \ie{}
a smooth genus 1 curve with a rational point $O_{\E}$, we
found two functions $x,y \in \K (\E)$ which are both regular everywhere
but at $O_{\E}$. These functions are related by the relation (\ref{eq:Weierstrass})
and hence the function $y^2 - x^3 -Ax - B$ vanishes everywhere on $\E$.
This leads to the following statement.

\begin{theorem}
  Let $\E$ be an elliptic over a field $\K$ of characteristic
  different from 2 and 3, \ie{} a smooth projective curve of genus 1
  with a rational point $O_{\E}$, then there exist $x, y \in \K (\E)$ such
  that the map
  \[
    \ratmap{\E}{\P^2}{P}{\left\{
        \begin{array}{ccl}
          (x(P) : y(P) : 1) &\text{if}& P \neq O_{\E}\\
          (0:1:0) & \text{if} & P = O_{\E}}
        \end{array}
        \right.
      \]
      induces an isomorphism from $\E$ to the projective closure of
      the projective curve of equation $Y^2 = X^3 + AXZ^2 + BZ^3$.
\end{theorem}

\begin{proof}
  The fact that the image of $\E$ is contained in such a curve is a
  consequence of the previous discussion. To prove that this map is
  actually an isomorphism and in particular that the target
  curve is smooth, we refer the reader to
  \cite[Prop.~III.3.1]{silverman2009book}.
\end{proof}

\begin{remark}
  Geometrically speaking, the sequence of changes of variables can be
  interpreted as follow. We started from an elliptic curve $\E$ and a
  first choice of functions $x, y$ in $\K (\E)$ lead to an isomorphism
  between $\E$ and a curve with equation (\ref{eq:eq1}). Then, we
  applied successive affine automorphisms to the plane in order to get
  curves of successive equations (\ref{eq:eq2}), (\ref{eq:eq3}) which
  are pairwise isomorphic and finish with a curve with
  equation~(\ref{eq:Weierstrass}) which is also isomorphic to $\E$.
\end{remark}

\begin{remark}
  In the sequel, we will not only consider Weierstrass form. Actually,
  one can show that any curve of equation
  \[
    y^2 = f(x)
  \]
  where $f$ is a squarefree polynomial of degree $3$ is an elliptic
  curve and there is a change of variables permitting to put it in
  Weierstrass form.
\end{remark}

\subsection{The $j$--invariant}
In what follows, it will be important to classify elliptic curves up to
isomorphism. For this sake, we introduce a fundamental invariant: the
{\em $j$--invariant}. Reconsider a Weierstrass equation
\eqref{eq:Weierstrass}
\[
  y^2 = x^3 + Ax + B.
\]
This equation is not unique, since once we got it, one can
still apply changes of variables of the form $y \leftarrow u^3 y,\
x \leftarrow u^2 x$ and dividing both sides by $u^6$. This leads
to another Weierstrass equation $y^2 = x^3 + A'x + B'$ where
$A' = \frac{A}{u^4}$ and $B' = \frac{B}{u^6}$.
Let us introduce
\[
  j \eqdef 1728 \frac{4A^3}{4A^3+27B^2}\cdot
\]
This quantity is well--defined since one can prove that
the denominator $4A^3 + 27 B^2$ is zero if and only if the corresponding curve
is singular (see \cite[Prop.~III.1.4(a)(i)]{silverman2009book}). Hence, $j$ is well--defined
for any elliptic curve since, by definition, such curves are smooth.
Moreover, $j$ is left invariant by the previous change of variables
and one can show that, once we obtained a Weierstrass equation,
the only changes of variables preserving the Weierstrass equation structure
are the aforementioned ones.

We conclude this subsection by the following statements asserting that
the $j$--invariant characterises an elliptic curve over $\Kbar$ in a
unique manner. The proof is omitted and can be found in
\cite[Prop.~III.1.4(b-c)]{silverman2009book}.

\begin{proposition}
  Two elliptic curves are isomorphic over $\Kbar$ if and only if
  they have the same $j$--invariant. Conversely, given $j_0 \in \Kbar$,
  there exists an elliptic curve $\E$ over $\K (j_0)$ with $j$--invariant
  $j_0$.
\end{proposition}

\begin{remark}
  Note that two elliptic curves defined over $\K$ may be isomorphic
  over $\Kbar$ without being isomorphic over $\K$. For instance,
  suppose that $-1$ is not a square in $\K$. Then, between the curves
  with equation
  \[
    y^2 = x^3+ Ax + B \quad \text{and} \quad -y^2 = x^3 + Ax + B
  \]
  are related by the isomorphism defined over $\Kbar$ given by
  $(x,y) \mapsto (x, \sqrt{-1}\ y)$ but there may not exist an
  isomorphism defined over $\K$. Such curves are said to be a {\em
    twist} of each other.
\end{remark}

\begin{remark}
  Starting from a $j$--invariant $j_0 \in \Kbar$, an explicit equation
  for an elliptic curve with this $j$--invariant is given by
  \[
    y^2 + xy = x^3 - \frac{36}{j_0-1728} x - \frac{1}{j_0 - 1728} \quad
    \text{if}\quad j_0 \neq 0, 1728
  \]
  and
  \[
    y^2+y = x^3 \ \ \text{if}\ \ j_0 = 0 \qquad \text{and}
    \qquad y^2 = x^3+x \ \ \text{if}
    \ \ j = 1278.
  \]
\end{remark}

\subsection{The group law}\label{ss:group_law}
A remarkable feature of such curves is that they naturally have a
group structure. Namely, given an elliptic curve $\E$ over $\K$, the
set $\E(\K)$ has a structure of abelian group. More generally, for any
algebraic extension $\LL$ of $\K$, then $\E (\LL)$ has an abelian
group structure too. This group structure is usually represented with
a so--called {\em chord and tangent process} as represented by Figure~\ref{fig:addition_law}.
It can be described as follows:
\begin{itemize}
\item[\textbullet] Given two points $P, Q \in \E(\LL)$, draw the line
  $\Lcal \subseteq \A^2$ (or $\P^2$) joining them. If $P=Q$
  let $\Lcal$ be the tangent line of $\E$ at $P$.
\item[\textbullet] Since the curve has degree $3$, its intersection
  with $\Lcal$ and $\E$ has $3$ points counted with multiplicity and
  hence either $\Lcal$ is vertical and then the third point is
  $R_0 = O_{\E}$ or denote by $R_0 = (x_{R_0}, y_{R_0})$ be the
  third\footnote{Possibly $R_0$ equals $P$ or $Q$. This is the reason
    why we mentioned 3 points {\em counted with multiplicity}.
    If the intersection multiplicity of $\Lcal$ with $\E$ at $P$ (resp. $Q$)
  is $2$ then, we set $R_0 \eqdef P$ (resp. $Q$).} point
  of intersection of this line with $\E$.
\item[\textbullet] Let $R$ be the point with coordinates
  $(x_{R_0}, -y_{R_0})$ if $R_0 \neq O_{\E}$ and the point $O_\E$
  otherwise.  This point is defined to be the {\em sum of $P$ and $Q$}.
\end{itemize}

\begin{figure}
  \centering
  \includegraphics[scale = .4]{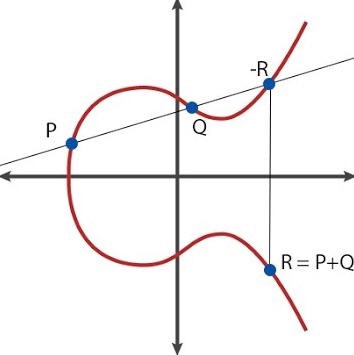}
  \caption{The addition law on an elliptic curve (Source:
      Cornelius Sch\"atz blog)}
  \label{fig:addition_law}
\end{figure}

\begin{exercise}\label{ex:group_law}
  \begin{enumerate}[(1)]
  \item\label{item:three_points} Prove that the intersection of $\E$ with a line
    is made of $3$ points of $\P^2$ possibly counted with multiplicity;
  \item Prove that $R_0 \in \E(\LL)$;
  \end{enumerate}
\end{exercise}

This description is simple to understand but it is not completely
obvious to prove that it provides a group structure. In particular,
the associativity is far from being obvious using this description.
Here we will show that this group structure can also be understood as
a group law inherited from that of some quotient of the divisor group.
Indeed, denote by $\Div_\K (\E)$ the group of rational divisors on
$\E$ and by $\Div_\K^0 (\E)$ the subgroup of divisors of degree
$0$. Finally denote by $\text{Princ}_\K (\E)$ the group of principal
divisors, \ie{} of divisors of the form $(f)$ where
$f \in \K (\E) \setminus \{0\}$.  From Remark~\ref{rem:princ=rat}
together with Proposition~\ref{prop:degree_principal},
$\text{Princ}_\K (\E)$ is a subgroup of $\Div_\K^0 (\E)$ and the quotient is
denoted
\[
  \Pic_\K^0 (\E) \eqdef \Div_\K^0(\E) / \text{Princ}_\K (\E).
\]

\begin{proposition}
  Let $\E$ be an elliptic curve over $\K$.
  Any element of $\Pic_\K^0 (\E)$ has a representative of the form $P - O_{\E}$
  where $P$ is some rational point in $\E(\K)$.
\end{proposition}

\begin{proof}
  Let $G \in \Div_\K^0 (\E)$. The divisor $G + O_{\E}$ has degree $1$ and, from
  Riemann--Roch Theorem $L(G+O_{\E})$ has dimension $1$. Thus, there
  exists $f \in L(G+O_{\E}) \setminus \{0\}$. By definition of $L(G + O_\E)$,
  the function $f$ satisfies
  \[
    (f) + G + O_{\E} \geq 0.
  \]
  The latter divisor is positive with degree $1$ and hence equals some
  rational point $P$. Thus $(f)+ G = P-O_{\E}$, which entails that $G$
  and $P-O_{\E}$ have the same class in $\Pic_\K^0(\E)$.
\end{proof}

\begin{theorem}
  Let $P, Q \in \E(\K)$ and $R$ be the sum of $P+ Q$ according to the
  previously introduced addition law. Then, the classes of $R-O_{\E}$
  and $(P-O_{\E}) + (Q-O_{\E})$ are the same in $\Pic_\K^0(\EE)$.
\end{theorem}

\begin{proof}
  From Exercise~\ref{ex:group_law}~(\ref{item:three_points}), there is
  a point $R_0$ which is contained in the line $\Lcal$ joining $P$ and
  $Q$.  Moreover $R, R_0$ are contained in a vertical line $\Lcal'$,
  the verticality entails that, projectively speaking, $R_0, R$ and
  $O_{\E}$ are in the projective closure of the line $\Lcal'$.  Let
  $H(X,Y,Z)$ and $H'(X, Y, Z)$ be homogeneous polynomials of degree
  $1$ providing equations of the projective closures of $\Lcal$ and
  $\Lcal'$ respectively. The rational function
  $h \eqdef \frac{H}{H'} \in \K (\E)$ has divisor
  \[
    (h) = (P+Q+R_0) - (R_0 + R + O_{\E}) = (P-O_{\E}) + (Q-O_{\E}) - (R - O_{\E}),
  \]
  which concludes the proof.
\end{proof}

As a conclusion, we have the following bijection:
\[
  \map{\E(\K)}{\Pic_\K^0 (\E)}{P}{P-O_{\E} \mod \text{Princ}_\K(\E)}.
\]
Via this bijection, we can equip $\E(\K)$ with a group structure
whose law is nothing but the previously described chord--tangent one.
Therefore, $\E(\K)$ equipped with the chord--tangent law has a group
structure which is isomorphic to $\text{Pic}_\K^0(\E)$.

\begin{remark}
  Here again, note that we discussed about the group structure
  of the set of rational points $\E(\K)$ but actually, for any algebraic
  extension $\LL/\K$, the set $\E(\LL)$ has also a group structure with
  $\E (\LL)$ as a subgroup. In particular, the whole set of geometric
  points $\E (\Kbar)$ has a structure of abelian group.
\end{remark}

\subsection{Torsion and isogenies}
Once we know that elliptic curves are equipped with an abelian group
structure it is of course natural to study the morphisms relating
these curves. For this sake, we first need to discuss some specific
subgroups of points of elliptic curves: their torsion subgroups.

\subsubsection{Torsion subgroups}
Given an elliptic curve $\E$ and an integer $\ell$, one is interested in the
group
\[
  \E [\ell] \eqdef \{P \in \E(\Kbar) ~|~ \ell P = 0\},
\]
where $\ell P$ means ``$P+ \cdots + P$'' (added $\ell$ times).
Interestingly, this group has a natural structure of $\Z / \ell \Z$--module,
and, in particular, is an $\F_\ell$--vector space when $\ell$ is prime.
The next theorem asserts that this space has always dimension $2$
when $\ell$ is prime to the characteristic. The proof of the next statement
is omitted.

\begin{theorem}\label{thm:torsion_of_elliptic}
  Let $\E$ be an elliptic curve over $\K$. Let $\ell$ be an integer.
  If $\ell$ is prime to the characteristic of $\K$, then
  \[
    \E [\ell] \simeq \Z / \ell \Z \times \Z / \ell \Z.
  \]
  Else, if $p$ denotes the characteristic of $\K$ and $p \neq 0$, then
  \[
    \E [p] \simeq \left\{
      \begin{array}{cc}
        \text{either} & \Z/p\Z  \\
        \text{or} & 0.
      \end{array}
      \right.
    \]
    In the former case the curve is said to be {\em ordinary}, in the latter
    it is said to be {\em supersingular}.
\end{theorem}

\subsubsection{Isogenies}
Given two elliptic curves $\E, \E'$, an isogeny
$\phi : \E \rightarrow \E'$ is a morphism between these curves sending
the neutral element $O_\E$ onto $O_{\E'}$. Such a map is always
surjective from $\E (\Kbar)$ into $\E'(\Kbar)$. A remarkable property
is that such a map is necessarily a morphism of groups (see
\cite[Thm.~III.4.8]{silverman2009book}.  As any morphism of curves, an
isogeny $\phi : \E \rightarrow \E'$ induces a field extension
$\K(\E')/\K(\E)$. The {\em degree} of the isogeny is the degree of the
field extension and the isogeny is said to be {\em separable} if the
field extension is separable too.  An isogeny of degree $\ell$ will
usually be referred to as {\em an $\ell$--isogeny}.

\begin{example}
  Taken from \cite[Ex.~III.4.5]{silverman2009book}.
  Let $a, b\in \K$, $b \neq 0$ and $a^2-4b \neq 0$. Consider the curves with equations:
  \begin{eqnarray*}
    \E : y^2 &=& x^3+ax^2+bx\\
    \E': y^2 &=& x^3-2ax^2 + (a^2-4b)x.
  \end{eqnarray*}
  The following map is a $2$--isogeny:
  \begin{equation}\label{eq:ex_iso}
    \map{\E}{\E'}{(x,y)}{\left(\frac{y^2}{x^2}, \frac{y(b-x^2)}{x^2} \right).}
  \end{equation}
\end{example}

\begin{exercise}
  Check that the map~(\ref{eq:ex_iso}) actually sends $\E$ into $\E'$.
  {\em Hint. Using a computer algebra software may be helpful for this
    exercise.}
\end{exercise}

\begin{example}
  Another example for isogenies of elliptic curves over a finite field $\Fq$
  of characteristic $p$ is the Frobenius map
  \[
    \map{\E}{\E^{(p)}}{(x,y)}{(x^p,y^p).}
  \]
  This isogeny is purely inseparable and sends the curve $\E$
  with Weierstrass equation $y^2 = x^3 + Ax + B$ onto
  the curve $\E^{(p)}$ of equation $y^2 = x^3 + A^p x + B^p$.
  If $\E$ is defined over $\Fq$ (\ie{} if $A, B \in \Fq$) then
  the Frobenius map is an endomorphism of $\E$.
\end{example}

\begin{example}
  For any $m > 0$ prime to the characteristic and any elliptic curve
  $\E$ over $\K$, the map $P \mapsto mP$ is an isogeny from $\E$ into
  itself. Its kernel is $\E[m]$. 
\end{example}

A separable isogeny of degree $\ell$, regarded as a group morphism
$\E(\Kbar) \rightarrow \E (\Kbar)$ is surjective with a finite kernel
of cardinality $\ell$. Its kernel is a
subgroup of $\E [\ell]$. 

\begin{theorem}[{\cite[Prop.~III.4.12]{silverman2009book}}]
  \label{thm:isogenies_as_quotients}
  For any finite subgroup $K \subseteq \E (\Kbar)$, there exists an elliptic
  curve $\E'$ defined over $\Kbar$ and an isogeny $\phi : \E \rightarrow \E'$
  such that $\ker \phi = K$. The curve $\E'$ is sometimes denoted
  as $\E/K$.
\end{theorem}

\begin{remark}
  In the previous statement, further precision can be given on the
  field of definition of $\E'$ and $\phi$. The field of definition of
  the group $K$ is the smallest extension $\LL/\K$ such that
  $K$ is globally invariant under the action of $\Gal (\Kbar /\LL)$.
  The field of definition of $\phi$ and $\E'$ is that of $K$.

  Note that the field of definition is {\bf not} the smallest field
  of definition of any geometric point of $K$. For instance,
  there may be non rational $m$--torsion points while $\E[m]$
  is defined over $\K$.
\end{remark}

Finally, even if an isogeny $\phi : \E \rightarrow \E'$ of degree $m>1$
is not an isomorphism in general, and hence has no inverse, it has a
so--called {\em dual isogeny} $\hat{\phi}$ which is the unique isogeny
$\hat{\phi} : \E' \rightarrow \E$ such that
\[
  \phi \circ \hat{\phi} : \map{\E}{\E}{P}{mP} \quad
  \text{and}\quad
  \hat{\phi}\circ \phi : \map{\E'}{\E'}{Q}{mQ.}
\]
The existence and uniqueness of this map are proven in
\cite[\S~III.6]{silverman2009book}.

\begin{example}
  In the case of a separable isogeny $\phi : \E \rightarrow \E'$ of
  degree $m$, its kernel is a group with $m$ elements. By Lagrange
  Theorem, such a finite group is of $m$--torsion and hence
  $\ker \phi \subseteq \E [m]$. Then $\phi (\E [m])$ is a finite
  subgroup and, from Theorem~\ref{thm:isogenies_as_quotients}, there
  is an isogeny $\varphi : \E' \rightarrow \E' / \phi(\E[m])$, which
  is nothing but the dual isogeny of $\phi$. In particular
  $\E'/ \phi(\E [m])\simeq \E/\E[m] \simeq \E$.
  The last isomorphism is induced by the map $P \mapsto mP$.
\end{example}

All the previously introduced notions: torsion, isogenies, dual isogenies
will be re--discussed and better illustrated in the subsequent section about
elliptic curves over $\C$. In this context, these notions will
be much easier to visualise.

\subsection{Elliptic curves over the complex numbers}
\label{ss:elliptic_over_C}
As already explained earlier, complex elliptic curves is not the topic
of this lecture. It is however necessary to discuss a bit about them.
In order not to spend too much time on the topic,
many proofs of non trivial statements are omitted and replaced by
precise references. Clearly, the summary to follow is strictly
included in Chapter VI of Silverman's book \cite{silverman2009book}.

\subsubsection{Lattices and the Weierstrass $\wp$ function}
In the complex setting, an elliptic curve is isomorphic to a complex
torus. Namely, a lattice of $\C$ is a discrete subgroup $\Lambda$ with compact
quotient and it is well--known that such a group is of the form
\[\Lambda = \Z \omega_1 \oplus \Z \omega_2\]
where $\omega_1, \omega_2$ are linearly independent over $\R$.
The relation between a torus $\C/ \Lambda$ and an elliptic curve
is far from being obvious and the key for connecting these two objects
is Weierstrass $\wp_\Lambda$ function defined as
\[
  \wp_{\Lambda}(z) \eqdef \frac{1}{z^2} + \sum_{\omega \in \Lambda \setminus \{0\}}
  \left(\frac{1}{(z-\omega)^2} - \frac{1}{\omega^2}\right)\cdot
\]
This is a meromorphic function with pole locus $\Lambda$ which is
$\Lambda$--periodic, \ie{} for any $z \in \C \setminus \Lambda$ and
$\omega \in \Lambda$, $\wp(z + \omega) = \wp(z)$. The proof of
convergence of the series is left to the reader.

Note that, since $\wp$ is $\Lambda$--periodic, it passes to the
quotient and induces a meromorphic function on the torus
$\C/ \Lambda$.  The function $\wp$ is fundamental in the sense that
actually, any $\Lambda$--periodic meromorphic function can be
expressed as a rational function in $\wp$ and its derivative $\wp'$ as
explained by the following statement.

\begin{theorem}
There exist complex numbers $g_2, g_3$,
  which depend on $\Lambda$ such that
  \[
    \forall z \in \C \setminus \Lambda,\qquad
    \wp_{\Lambda}'(z)^2 = 4 \wp_{\Lambda}(z)^3 + g_2 \wp_{\Lambda}(z)
    + g_3.
  \]
\end{theorem}

\begin{proof}
The series
  \[
    \wp(z) - \frac{1}{z^2} = \sum_{\omega \in \Lambda \setminus \{0\}}
  \left(\frac{1}{(z-\omega)^2} - \frac{1}{\omega^2}\right)
\]
is even and vanishes at $0$. Hence, in the neighbourhood of $0$, its
Taylor series expansion depends only on $z^2$. Thus, we deduce that
$\wp_{\Lambda}$ has a Laurent series expansion at $0$ of the form
\[
  \wp_{\Lambda}(z) = \frac{1}{z^2} +  O(z^2),
\]
and
\[
  \wp'_{\Lambda}(z) = -\frac{2}{z^3} + O(z).
\]
Therefore, in a neighbourhood of $0$, 
$\wp'_{\Lambda}(z)^2 - 4 \wp_\Lambda(z)^3 = O(\frac{1}{z^2})$ and
there is a constant $g_2 \in \C$ such that
\begin{equation}\label{eq:poles_eqdif_Weierstrass}
  \wp'_{\Lambda}(z)^2 - 4 \wp_\Lambda(z)^3 -g_2 \wp_{\Lambda}(z) =
  O(1).
\end{equation}
The function
$\wp'_{\Lambda}(z)^2 - 4 \wp_\Lambda(z)^3 -g_2 \wp_{\Lambda}(z)$ is
$\Lambda$--periodic, meromorphic on $\C$ with pole locus contained in
$\Lambda$. From~(\ref{eq:poles_eqdif_Weierstrass}), it has no pole at
$0$ and, by $\Lambda$--periodicity has no pole at all and hence is
holomorphic on $\C$. Since it continuous and $\Lambda$--periodic on
$\C$, it is bounded, and by Liouville's theorem, it should be
constant. Therefore, there exists $g_3 \in \C$ such that
$\wp_{\Lambda}'(z)^2 = 4 \wp_{\Lambda}(z)^3 + g_2 \wp_{\Lambda}(z) +
g_3$.
\end{proof}

\begin{exercise}
  Prove that a $\Lambda$--periodic holomorphic function is bounded on
  $\C$.
\end{exercise}

A finer analysis of the series permits to estimate $g_2, g_3$ in terms
of $\Lambda$ and to prove that the equation $y^2 = 4x^3 + g_2 x + g_3$
is that of a smooth curve, and hence of an elliptic curve.  With this
theorem at hand, we deduce the existence of a map from the torus
$\C/\Lambda$ into the elliptic curve $\E$ of equation
$y^2 = 4x^3+g_2 x + g_3$:
\begin{equation}\label{eq:Psi}
  \Psi_{\Lambda} : \map{\C/\Lambda}{\E}{z}{(\wp_{\Lambda}(z) :
    \wp'_{\Lambda}(z) : 1).}
\end{equation}
Note that this map is well--defined everywhere, since at $0$ which is a
pole of order $2$ of $\wp_{\Lambda}$ and of order $3$ of
$\wp'_{\Lambda}$ one can renormalise as
$(z^3\wp_{\Lambda(z)} : z^3 \wp'_{\Lambda}(z) : z^3)$ and evaluate at
$0$, which yields the point $O_{\E} = (0:1:0)$.  The following statement
gathers several nontrivial and fundamental results on complex tori: it
states a one-to-one correspondence between elliptic curves and complex
tori when regarded as complex varieties {\bf but also as groups}.

\begin{theorem}
  The map $\Psi_{\Lambda}$ defined in \eqref{eq:Psi} is a
  biholomorphic isomorphism between $\C/\Lambda$ and the elliptic
  curve $\E$ of equation $y^2 = 4x^3 + g_2 x + g_3$. Moreover, it also
  induces a group isomorphism from $\C/ \Lambda$ equipped with
  the addition law inherited from that of $\C$ into $\E(\C)$ equipped
  with its group law introduced in \S~\ref{ss:group_law}.
  Conversely, given any elliptic curve $\E_0$ over $\C$,
  there exists a lattice $\Lambda_0 \subset \C$ such that
  $\E_0$ is isomorphic to $\C / \Lambda_0$ via the map $\Psi_{\Lambda_0}$.
\end{theorem}

\begin{proof}
  See \cite[Prop.~VI.3.6]{silverman2009book} for the group
  isomorphism.  For the construction of a lattice from an elliptic
  curve, see \cite[\S~VI.1]{silverman2009book}.
\end{proof}

\subsubsection{Torsion, isogenies}
An interest of the complex setting is that the previous results on
torsion and isogenies are pretty easy to understand when regarding
elliptic curves as complex tori.

Let us start with the torsion. From
Theorem~\ref{thm:torsion_of_elliptic}, for $m$ prime to the
characteristic, the $m$--torsion of an elliptic curve is isomorphic
to $\Z/m\Z \times \Z /m\Z$.  In the complex setting, consider a torus
$\C / \Lambda$. Then, the torsion points correspond to points
$z \in \C$ such that $mz \in \Lambda$ and hence they correspond to the
points of the lattice $\frac 1 m \Lambda \supset \Lambda$. Then, the
torsion subgroup of $\C / \Lambda$ is isomorphic to
$\frac 1 m \Lambda / \Lambda$. Since $\Lambda$ is of the form
$\Z \omega_1 \oplus \Z \omega_2$ for some $\R$--independent elements
$\omega_1, \omega_2 \in \C$, we deduce that
\[
  (\C/\Lambda) [m] \simeq \left(\frac 1 m \Lambda\right) / \Lambda = \frac{\Z
    \frac{\omega_1}{m} \oplus \Z \frac{\omega_2}{m}}{\Z \omega_1
    \oplus \Z \omega_2} \simeq \Z/m\Z \oplus \Z/m\Z.
\]

Now consider isogenies. When considering
complex tori, isogenies are holomorphic maps
$\C/ \Lambda \rightarrow \C / \Lambda'$.  It turns out that such maps lift to $\C$ and
have a very particular structure.

\begin{theorem}\label{thm:elliptic_prop}
  Let $\Lambda, \Lambda' \subset \C$ be two lattices and
  $f : \C/ \Lambda \rightarrow \C / \Lambda'$ be a holomorphic map $0$
  sending onto $0$.  Then $f$ lifts to a holomorphic map
  $f_0 : \C \rightarrow \C$ such that
  \[
    \forall z \in \C,\quad
    f_0(z) \mod \Lambda' = f(z \mod \Lambda). 
  \]
  Moreover, $f_0$ is a similitude, \ie{} there exists $a \in \C$ such that
  \[\forall z \in \C,\ f_0(z) = az.\]
\end{theorem}

\begin{proof}
  See \cite[Thm.~VI.4.1]{silverman2009book}.
\end{proof}

With this statement at hand, we deduce that an isogeny
$\phi : \C / \Lambda \rightarrow \C/\Lambda'$ is induced by a map
$z \mapsto az$ with $a \Lambda \subseteq \Lambda'$.

\begin{example}
For instance, consider the lattices
\[
  \Lambda = \Z \oplus  \Z 2i \quad \text{and} \quad
  \Lambda' = 2\Z  \oplus \Z 2i
\]
then we easily see that the map $z \mapsto 2z$ induces an isogeny
$\C/ \Lambda \rightarrow \C / \Lambda'$. 
\end{example}

From a similitude $z \mapsto az$ such that
$a \Lambda \subset \Lambda'$, the degree of the corresponding isogeny
is given by $\card{(\Lambda' / a \Lambda)}$. In the previous example,
the isogeny has degree $2$.

Finally, let $\ell$ be a prime integer, and suppose that we have a
degree--$\ell$ isogeny
$\phi_1 : \C/ \Lambda \rightarrow \C / \Lambda'$.  This entails the
existence of $a \in \C$ such that
$\card{(\Lambda' / a \Lambda)} = \ell$. From the structure theorem of
finitely generated modules over principal ideal rings, we deduce the
existence of $\omega_1, \omega_2 \in \C$ such that
\[
  \Lambda = \Z \frac{\ell \omega_1}{a} \oplus \frac{\omega_2}{a}, \quad
  \Lambda' = \Z \omega_1 \oplus \Z \omega_2
\]
and $\phi_1$ is induced from the similitude $z \mapsto az$.

\begin{exercise}
  Prove the last assertion.
\end{exercise}

Now consider the map $z \mapsto \frac{\ell}{a}z$. It induces an isogeny
$\phi_2 : \C/ \Lambda' \rightarrow \C / \Lambda$
of degree $\ell$. Moreover
the composition of the two isogenies : 
\[
  \phi_2 \circ \phi_1 : \C /\Lambda \rightarrow \C/\Lambda' \rightarrow
  \C / \Lambda''
\]
is defined by $z \mod \Lambda \mapsto \ell z \mod \Lambda$ and hence is
nothing but the multiplication by $\ell$ in $\C / \Lambda$. Therefore,
$\phi_2$ is nothing but the dual isogeny map $\hat{\phi}_1$ of $\phi_1$.

\subsection{Automorphisms}
Now, we have a nice description of morphisms of complex elliptic
curves. Moreover, an endomorphism of an elliptic curve, or
equivalently of a complex torus $\C/ \Lambda$, is induced by a
similitude $z \mapsto az$ such that $a\Lambda \subset \Lambda$.

One will also be interested in the sequel by automorphisms of an
elliptic curve, which correspond to similitudes
$z \mapsto az$ such that $a\Lambda = \Lambda$. One can prove that
such an $a$ satisfies $|a| = 1$.

\begin{remark}
  Clearly for any integer $N$ and any lattice $\Lambda$ we have
  $N \Lambda \subset \Lambda$ and the map $z \mapsto Nz$ induces an
  endomorphism of $\C / \Lambda$ which is the multiplication by $N$
  map. In addition, if there exists $a \in \C \setminus \Z$ such
  that $a \Lambda \subset \Lambda$, then the corresponding elliptic
  curve is said to be {\em with complex multiplication}.
\end{remark}

An elementary automorphism for any complex torus is $z \mapsto -z$.
Back to the map $\Psi_{\Lambda}$ in (\ref{eq:Psi}) and using the
fact that $\wp_{\Lambda}$ and $\wp'_{\Lambda}$ are respectively
even and odd, we deduce that this map corresponds on the elliptic curve
to the symmetry with respect to the $x$--axis:
\[(x,y) \longmapsto (x, -y).\]

Furthermore, some sporadic elliptic curves have nontrivial automophisms
coming from $z \mapsto az$ with $|a|=1$ and $a \notin \{\pm 1\}$.

\begin{theorem}\label{thm:automorphisms}
  Let $\C / \Lambda$ be a complex torus with an automorphism $z \mapsto az$
  with $|a| = 1$ and $a \notin \{\pm 1\}$. Equivalently, the lattice $\Lambda$
  satisfies $\Lambda = a \Lambda$. Then, $\Lambda$ is the image
  by a similitude of one of these two lattices:
  \[
    \Z \oplus \Z i \quad \text{or} \quad \Z \oplus \Z \rho,
  \]
  where $\rho = e^{\frac{i\pi}{3}}$.
\end{theorem}

\begin{proof}
  Let $\Lambda$ be a lattice such that $a\Lambda = \Lambda$ and
  $\nu \in \Lambda \setminus \{0\}$ be a vector of minimal modulus.
  Since we look for $\Lambda$ up to a similitude, one can assume that
  $\nu = 1$ and that for all $\omega \in \Lambda\setminus \{0\}$,
  $|\omega| \geq 1$.
  Assuming that $1 \in \Lambda$, then, by assumption on
  $\Lambda$, we deduce that $a, a^2$ are elements of $\Lambda$ too.
  Since $a \notin \R$, its minimal polynomial over $\R$ is
  \[
    \begin{array}{ccl}
      (x-a)(x-\bar a) & = & x^2 + 2 \text{Re}(a)x + |a|^2 \\
      & = & x^2 + 2
    \text{Re}(a)x + 1, 
    \end{array}
  \]
  where $\text{Re}(a)$ denotes the real part of $a$. 
  Consequently,
  \[
    a ^2 + 1= -2 \text{Re}(a) a.
  \]
  Note that $|a| = 1$ and $a \notin \R$ entails
  $-1 < \text{Re}(a) < 1$.  If $2\text{Re}(a) \notin \Z$, then there
  is $\varepsilon \in \{-1, 0, 1\}$ such that
  \[
    a^2 + \varepsilon a + 1 = \gamma a
  \]
  for some $0 < \gamma < 1$. Since the left--hand side is a
  $\Z$--linear combination of elements of $\Lambda$, then
  $\gamma a \in \Lambda$ which contradicts the assumption that any
  nonzero $\omega \in \Lambda$ satisfies $|\omega| \geq 1$.
  Therefore $ \text{Re}(a) \in \{- \frac 1 2, 0, \frac 1 2\}$.
  Case $\text{Re}(a) = 0$ provides the case
  \(\Lambda = \Z \oplus \Z i\) and the two other cases provide
  the same lattice, namely $\Z \oplus \Z \rho$.
\end{proof}

The corresponding elliptic curves can be proved to have respective equations:
\begin{equation}\label{eq:curve_with_auto}
  \begin{array}{cclcccl}
    y^2 & = & x^3 + x & \qquad \text{for} \qquad & \Lambda & = & \Z \oplus
                                                                 \Z i \quad \text{($j$--invariant $1728$)}\\
    y^2 & = & x^3 + 1 & \qquad \text{for} \qquad & \Lambda & = & \Z \oplus
                                                                 \Z \rho
                                                                 \quad \text{($j$--invariant $0$).}
  \end{array}
\end{equation}
The corresponding automorphisms being respectively
\[
  \begin{array}{ccl}
    (x,y) & \longmapsto & (-x, iy) \\
    (x,y) & \longmapsto & (\rho x, -y).
  \end{array}
\]
Note that these automorphisms have respective orders $4$ and $6$ which
are the multiplicative orders of $i$ and $\rho$.  Finally, note that
for any field containing fourth and sixth roots of $1$, the curves
with equations~\eqref{eq:curve_with_auto} have a nontrivial
automorphism group. Moreover, one can prove that they are the only
curves with non trivial automorphism groups
\cite[Thm.~III.10.1]{silverman2009book} and that their automorphism groups
have respective cardinalities $4$ and $6$.

 \section{Modular curves}\label{sec:modular}

\subsection{The Poincaré upper half plane}
The objective is to classify elliptic curves over $\C$ up to
isomorphism. As explained in \S~\ref{ss:elliptic_over_C}, this reduces
to classify lattices up to similitudes whose definition
is recalled there.

\begin{definition}[Similitudes of $\C$]
  A {\em similitude} of $\C$ is a map of the form $z \mapsto az$
  for some $a \in \C^\times$.
\end{definition}

Besides the action of the group of similitudes on the set of lattices
of $\C$, lattices are described by a basis which is not unique. This
requires to introduce another group action on the possible bases.
Namely, given a lattice
\[
\Lambda = \Z \omega_1 \oplus \Z \omega_2,
\]
the basis $(\omega_1, \omega_2)$ is not unique and any other basis
$(\mu_1, \mu_2)$ is deduced from $(\omega_1, \omega_2)$ by
\[
  \begin{pmatrix}
    \mu_1 \\ \mu_2
  \end{pmatrix} = M \cdot
  \begin{pmatrix}
    \omega_1 \\ \omega_2
  \end{pmatrix},
  \quad \text{for\ some\ } M \in \GL_2(\Z).
\]
Up to swapping the entries of the basis, one can always assume that
the bases we consider have the same orientation, \ie{} that
$\text{Im}(\frac{\omega_1}{\omega_2})>0$
(resp. $\text{Im}(\frac{\mu_1}{\mu_2})>0$), where $\text{Im}(\cdot)$
denotes the imaginary part of a complex number.  If the bases are
chosen under this constraint, then the transition matrix $M$ always
has a positive determinant and hence is in $\SL$.  Therefore, the set
of lattices of $\C$ is in one-to-one correspondence with the
classes of pairs $(\omega_1, \omega_2) \in \C^2$ with
$\text{Im}(\frac{\omega_1}{\omega_2}) > 0$ modulo the action of $\SL$.
Next, we need to consider the action of similitudes. Starting from
$\Lambda = \Z\omega_1 \oplus \Z \omega_2$ with
$\text{Im}(\frac{\omega_1}{\omega_2}) > 0$ and applying the similitude
$z \mapsto \frac{1}{\omega_2} z$, we get a similar lattice:
\[
  \Z \oplus \Z \tau
\]
with $\tau = \frac{\omega_1}{\omega_2}$ and hence $\text{Im}(\tau)>0$.
Let
\[
  \HP \eqdef \left\{z \in \C ~|~ \text{Im}(z) > 0 \right\},
\]
be the {\em Poincaré upper half plane}.  Then any lattice up to
similitude can be associated to an element $\tau \in \HP$ and the action of
$\SL$ on bases of lattices induces the following action on $\HP$.
Starting from
\[
  M  =
  \begin{pmatrix}
    a & b \\ c & d
  \end{pmatrix} \in \SL,
\] 
$M$ acts on bases as:
\[
  M\cdot \begin{pmatrix}
    \omega_1 \\ \omega_2
  \end{pmatrix}
  =
  \begin{pmatrix}
    a \omega_1 + b\omega_2 \\
    c \omega_1 + d\omega_2
  \end{pmatrix}.
\]
Therefore since $\tau = \frac{\omega_1}{\omega_2}$, we naturally define the
action of $\SL$ on $\HP$ by
\begin{equation}\label{eq:action_SL2}
  M\cdot \tau \eqdef \frac{a\omega_1 + b\omega_2}{c\omega_1 + d\omega_2}=
  \frac{a\tau + b}{c\tau + d}\cdot
\end{equation}
{\bf In summary,} according to the discussion of \S~\ref{ss:elliptic_over_C},
we have the
following correspondence:

\bigskip

\[
  \begin{array}{ccccccc}
    \text{Elliptic curves} && \text{Complex tori}& & \text{Lattices of }\C & & \text{Points of }\HP \\
    \text{up to} & \longleftrightarrow & \text{up to}& \longleftrightarrow & \text{up to} & \longleftrightarrow & \text{modulo}\\
    \text{isomorphism} && \text{biholomorphic}& & \text{similitudes} & & \text{the action (\ref{eq:action_SL2}) of} \\
    && \text{isomorphisms} & & && \SL
  \end{array}  
\]
\smallskip

Moreover, fundamental domains for the action of $\SL$ on
$\HP$ are represented in Figure~\ref{fig:modular_action},
which is a famous picture that you can find in so
many books of geometry or number theory.

\begin{figure}
  \centering
  \includegraphics[scale = .16]{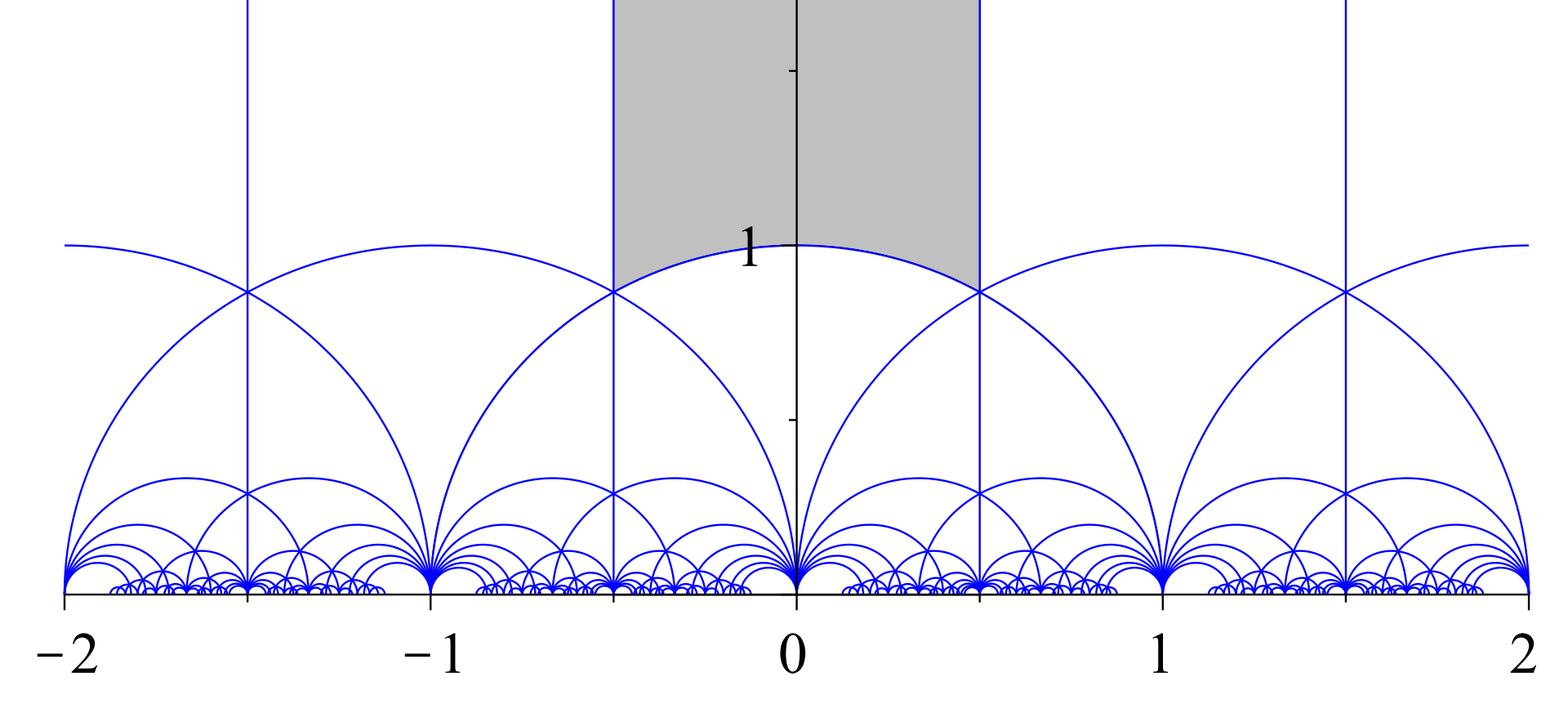}
  \caption{Fundamental domain for the action of $\SL$ on $\HP$
    (Source: Wikipedia)}
  \label{fig:modular_action}
\end{figure}

\subsection{The curve $X_0(1)$}
So, to parameterise the set of elliptic curves up to isomorphism, we
can consider the quotient $\SL \backslash \HP$. It is proved in
\cite[Prop.~2.21]{milneMF} that
this quotient is a complex variety isomorphic to $\A^1$, \ie{} to the
complex affine line. This is not surprising,
Theorem~\ref{thm:elliptic_prop} entails that complex elliptic curves
up to ismomorphisms are in one-to-one correspondence with $\C$ via
the map $\E \mapsto j(\E)$, where $j(\E)$ denotes the $j$--invariant
of $\E$.

Next, for convenience and in order to apply results on algebraic
curves introduced in \S~\ref{sec:curves}, it will be useful to have
some projective closure of this parameterising curve. In the complex
setting, this is nothing but a compactification and the affine line
can be compactified with one point. However, for a reason which will
appear to be more natural in the sequel, the compactification will be
made via a somehow more complicated construction.

The idea is to join to $\HP$ all the elements of $\Q$ which lie on the
boundary of $\HP$ together with a point at infinity. Namely, we define
\[
  \HP^* \eqdef \HP \cup \P^1(\Q).
\]
Next let us see how the action of $\SL$ extends to $\P^1(\Q)$.

\begin{proposition}\label{prop:cusps}
  Consider the following action of $\SL$  on $\P^1(\Q)$:
  \[
    \forall M =
    \begin{pmatrix}
      a & b \\ c & d
    \end{pmatrix} \in \SL,\
    (u:v) \in \P^1(\Q),\quad
    M\cdot (u:v) = (au+bv : cu+dv).
  \]
  This action is transitive, \ie{} for any $(u:v), (u':v') \in \P^1(\Q)$,
  there exists $M \in \SL$ such that $M\cdot (u:v) = (u':v')$.
\end{proposition}

\begin{proof}
  First, let us prove that the orbit of $(0:1)$ equals the whole
  $\P^1 (\Q)$. Note first that
  \[
    \begin{pmatrix}
      0 & -1 \\ 1 & 0
    \end{pmatrix} \cdot (0:1) = (-1:0) = (1:0).
  \]
  Hence $(1:0)$ is in the orbit of $(0:1)$.  Next, consider any other
  point $(s:t) \in \P^1(\Q) \setminus \{(1:0)\}$, \ie{} such that
  $t \neq 0$.  After multiplying the coordinates by a common
  denominator, one can suppose that $s, t \in \Z$ and after possibly
  dividing by their greatest common denominator, one can suppose $s,t$
  are prime to each other. By Bézout's Theorem, there exist $u, v \in \Z$
  such that $su + tv = 1$ and then
  \[
    \begin{pmatrix}
      t & u \\ -s & v
    \end{pmatrix}\cdot (0 : 1) = (u:v)\quad \text{and}\quad
    \begin{pmatrix}
      t & u \\ -s & v
    \end{pmatrix} \in \SL.
  \]
  Therefore, any element of $\P^1(\Q)$ is in the orbit of
  $(0:1)$. Finally, given two elements $(u:v),(u':v') \in \P^1(\Q)$
  there exist $M,M'$ such that $(u:v) = M \cdot (0:1)$ and
  $(u':v') = M' \cdot (0:1)$ and $(u':v') = M'M^{-1}(u:v)$.
\end{proof}

Therefore, the quotient
\(
  \SL \backslash \HP^*
  \) is nothing but the compactification of $\SL \backslash \HP$ by adjoining
  a single point. This quotient is usually denoted as
  $X_0(1)$ and is nothing but the Riemann sphere $\P^1 (\C)$.

  In terms of functions on $X_0(1)$, there exists a holomorphic
  function $j : \HP \rightarrow \C$ which is invariant under the
  action of $\SL$ and such that the induced map
  $\SL \backslash \HP \rightarrow \C$ is bijective. This map realises
  an isomorphism between $X_0(1)$ and $\P^1(\C)$. It can be
  ``made explicit'' as follows. From $\tau \in \HP$ construct the lattice
  $\Lambda_{\tau} = \Z \oplus \Z \tau$. Then using the Weierstrass
  $\wp_{\Lambda_\tau}$ function, compute an equation of the elliptic
  curve corresponding to $\C/ \Lambda_\tau$. Then, $j (\tau)$ is
  nothing but the $j$--invariant of this latter elliptic curve.

  \subsection{The curve $X_0 (\ell)$}
  Once we have a curve parameterising elliptic curves up to
  isomorphisms, we are still a bit far from our objective since we
  look for a family of curves whose sequence of genera goes to
  infinity, while we only got $\P^1$ which has genus $0$. To get
  curves with a higher genus, we need to enhance the structure and the
  idea is not only to classify elliptic curves up to isomorphism but
  to classify for a fixed integer $\ell$, the $\ell$--isogenies
  $\E \rightarrow \E'$ up to isomorphism. In the sequel we are only
  interested in the case where $\ell$ is prime (but many of the
  results to follow extend to an arbitrary degree of isogeny).

  \begin{remark}
    Note that, here, by ``up to ismomorphism'' we mean that two
    isogenies $\phi_1 : \E_1 \rightarrow \E_1'$ and
    $\phi_2 : \E_2 \rightarrow \E_2'$ will be said to be {\em
      isomorphic} if there exist two isomorphisms
    $\eta : \E_1 \rightarrow \E_2$ and $\nu : \E_1' \rightarrow \E_2'$
    such that the following diagram commutes.
    \begin{center}
      \begin{tikzpicture}
        \node at (0, 2) {$\E_1$};
        \node at (0, 0) {$\E_1'$};
        \node at (2, 2) {$\E_2$};
        \node at (2, 0) {$\E_2'$};
        \node at (-.2, 1) {${\eta}$};
        \node at (2.2, 1) {$\nu$};
        \node at (1, 2.3) {$\phi_1$};
        \node at (1, 0.3) {$\phi_2$};
        \draw[->] (0.3, 0.05) to (1.7, 0.05);
        \draw[->] (0.3, 2.05) to (1.7, 2.05);
        \draw[->] (0.05, 1.7) to (0.05, 0.3);
        \draw[->] (1.95, 1.7) to (1.95, 0.3);
      \end{tikzpicture}
    \end{center}
  \end{remark}

  \noindent From Theorem~\ref{thm:isogenies_as_quotients}, an
  $\ell$--isogeny $\E \rightarrow \E'$ corresponds to a pair
  $(\E, C)$ where $C \subseteq \E[\ell]$ is a subgroup of cardinality
  $\ell$. Then, in the complex setting, it reduces to classify pairs
  of lattices $\Lambda, \Lambda'$ such that
  $\Lambda \subseteq \Lambda'$ and $\card{(\Lambda' / \Lambda)} = \ell$.
  The structure theorem for finitely generated modules over a
  principal ideal ring asserts that there exists a basis
  $\omega_1, \omega_2$ of $\Lambda$ such that
  \[
    \Lambda = \Z \omega_1 \oplus \Z \omega_2 \quad \text{and} \quad
    \Lambda' = \Z \frac{\omega_1}{\ell} \oplus \Z \omega_2\cdot
  \]
  With the above description, one sees easily that
  $\E [\ell] = (\frac{1}{\ell}\Lambda) / \Lambda \simeq \F_\ell \oplus
  \F_\ell$ and $\Lambda' / \Lambda$ identifies to an
  $\F_\ell$--subspace of dimension $1$ of $\E [\ell]$, namely the
  subspace spanned by the class of $\frac{\omega_1}{\ell}$.  Since we
  wish to classify elliptic curves $\E$ with a given $\ell$--torsion
  subgroup $C$, we need to classify changes of basis preserving this
  subgroup. Observe that the action of $\SL$ on bases of $\Lambda$
  induces a natural action of $\SLl$ on
  $\E[\ell] = (\frac{1}{\ell}\Lambda)/\Lambda$. The elements of $\SLl$
  that fix the class of $\frac{\omega_1}{\ell}$ are the upper
  triangular matrices.  This motivates the definition of the {\em
    congruence subgroup} $\Gamma_0(\ell) \subset \SL$ defined as
\[
  \Gamma_0(\ell) \eqdef \left\{
    \begin{pmatrix}
      a & b \\ c & d
    \end{pmatrix} \in \SL ~\bigg|~  c \equiv 0 \mod \ell 
\right\}.
\]
Namely, this is the group of elements of $\SL$ which induce an automorphism
of $\E[\ell]$ fixing $C$.

\begin{exercise}
  Prove that the canonical map
  \[
    \SL \longrightarrow \SLl
  \]
  given by the reduction of the coefficients modulo $\ell$
  is surjective. To do it:
  \begin{enumerate}[(a)]
  \item Prove that an element of $\SLl$ has a lift $
    \begin{pmatrix}
      a & b \\ c & d
    \end{pmatrix}
    $ with $a,b,c,d \in Z$ such that $a, b$ are nonzero and prime to
    each other.
  \item Prove that for such a lift, $c, d$ can be replaced by $c', d'$
    such that $c \equiv c' \mod \ell$ and $d \equiv d' \mod \ell$
    so that $\det
    \begin{pmatrix}
      a & b \\ c' & d'
    \end{pmatrix} = 1.  $
  \end{enumerate}
\end{exercise}

Therefore, the set of $\ell$--isogenies between elliptic curves up to isomorphism
is in one-to-one correspondence with the complex variety
\[
  \Gamma_0(\ell) \backslash \HP.
\]
This variety has a compactification
\[
  X_0(\ell) \eqdef \Gamma_0(\ell) \backslash \HP^*. 
\]
This is a compact Riemann surface and it can be proved that such an
object is actually algebraic, \ie{} is biholomorphic with a smooth
complex projective curve.  This structure of algebraic curve is
discussed further.

The next statement gives a crucial information, namely
the genus of $X_0(\ell)$.

\begin{theorem}\label{thm:genus_X0ell}
  For a prime number $\ell> 3$,
  the genus $g_{\ell}$ of $X_0(\ell)$ equals
  \[
   g_{\ell} = \left\{
     \begin{array}{ccll}
       \frac{ \ell - 1}{ 12}-1 & \text{if} & \ell \equiv 1 &\mod [12]\\
       \frac{ \ell - 5}{ 12} & \text{if} & \ell \equiv 5 &\mod [12]\\
       \frac{ \ell - 7}{ 12} & \text{if} & \ell \equiv 7 &\mod [12]\\
       \frac{ \ell + 1}{ 12} & \text{if} & \ell \equiv 11 &\mod [12].\\       
     \end{array}
     \right.
  \]
\end{theorem}

We first need two technical lemmas.

\begin{lemma}\label{lem:isomorphic_subgroups}
  Let $\ell$ be a prime integer.
  Let $\Lambda \subseteq \C$ be a lattice and $\Lambda_1, \Lambda_2$
  be two distinct lattices both containing $\Lambda$ and
  $\card{(\Lambda_1 / \Lambda)} = \card{(\Lambda_2 / \Lambda)} = \ell$.
  Suppose that $a \Lambda_1 = \Lambda_2$ for some $a \in \C$. Then
  $|a| = 1$ and $a\Lambda = \Lambda$.
  Equivalently, given an elliptic curve $\E$ over $\C$ and two distinct
  subgroups $C_1, C_2$ of cardinality $\ell$ of $\E [\ell]$. If the
  curves $\E/C_1$ and $\E/C_2$ are isomorphic, then there is an automorphism
  of $\E$ sending $C_1$ onto $C_2$.
\end{lemma}

\begin{remark}
  Note that if $\E$ has such an automorphism, then it should be one of
  the two curves mentioned in Theorem~\ref{thm:automorphisms}.
\end{remark}

\begin{proof}[Proof of Lemma~\ref{lem:isomorphic_subgroups}]
  {\bf Step 1. An adapted basis.}
  We claim that there exists $\omega_1, \omega_2 \in \C$ such that
  \begin{equation}\label{eq:bases}
    \Lambda = \Z \omega_1 \oplus \Z \omega_2 \quad
    \text{and}\quad
    \Lambda_1 = \Z \frac{\omega_1}{\ell} \oplus \Z \omega_2
    \quad
    \text{and}\quad
    \Lambda_2 = \Z \omega_1 \oplus \Z \frac{\omega_2}{\ell}\cdot 
  \end{equation}
  The existence of $\omega_1, \omega_2$ can be obtained as follows. First,
  the structure theorem for finitely generated modules over principal
  ideal rings asserts the existence of a basis $\eta_1, \eta_2$ such that
  \[
    \Lambda = \Z \eta_1 \oplus \Z \eta_2 \quad \text{and}
    \quad \Lambda_1 = \Z \frac{\eta_1}{\ell} \oplus \Z \eta_2.
  \]
  Next, we claim that
  \[
    \Lambda_2 = \Z \eta_1 \oplus \Z \frac{u\eta_1 + \eta_2}{\ell},
  \]
  for some $u \in \{0, \dots, \ell-1\}$. Indeed, consider
  $\left(\frac{1}{\ell} \Lambda \right) / \Lambda$, which isomorphic to
  $\F_\ell \times \F_\ell$. In this quotient, $\Lambda_1 / \Lambda$ and
  $\Lambda_2 / \Lambda$ are identified to two $\F_\ell$--subspaces of
  dimension $1$ in direct sum. The subspace $\Lambda_1 / \Lambda$ is
  spanned by the class of $\frac{\eta_1}{\ell}$ and the fact that
  $\Lambda_1 / \Lambda$ and $\Lambda_2 / \Lambda$ are in direct sum in
  $\left(\frac{1}{\ell} \Lambda \right) / \Lambda$ entails that
  $\Lambda_2 / \Lambda$ should be spanned by the class of
  $\frac{u \eta_1 + \eta_2}{\ell}$ for some $u \in \F_\ell$. This
  implies that there exists $u \in \{0, \dots, \ell- 1\}$ such that
  $\frac{u\eta_1 + \eta_2}{\ell} \in \Lambda_2$ and hence
  \[
    \Z \eta_1 \oplus \Z \frac{u \eta_1 + \eta_2}{\ell} \subseteq \Lambda_2.
  \]
  Then, since $\card{(\Lambda_2 / \Lambda)} = \ell$, we can deduce that
  the above inclusion is actually an equality.
  Finally, define $\omega_1, \omega_2$ as
  \[
    \begin{pmatrix}
      \omega_1 \\ \omega_2
    \end{pmatrix}
     \eqdef   \begin{pmatrix}
    1 & u \\ 0 & 1
  \end{pmatrix}
  \cdot
    \begin{pmatrix}
      \eta_1 \\ \eta_2
    \end{pmatrix}.  
  \]
  Note that the above change of variables is given by a matrix in
  $\SL$ and provides a basis for $\Lambda$ which
  satisfies~(\ref{eq:bases}).

\medskip

\noindent {\bf Step 2. The modulus of $a$.}
A classical notion in lattice theory is that of the {\em determinant}
or {\em volume} of the lattice. It can be defined as follows.
Consider $\Lambda = \Z \omega_1 \oplus \Z \omega_2$ and regard $\C$ as
a $2$--dimensional $\R$--vector space with canonical basis $(1, i)$.
Since any basis of $\Lambda$ can be deduced from
$(\omega_1, \omega_2)$ by applying a matrix in $\GL_2(\Z)$, {\em i.e.}
a matrix with determinant $\pm 1$, the quantity
$|\det(\omega_1,\omega_2)|$ is the same for any basis of $\Lambda$.
Hence we denote this quantity $|\det \Lambda|$.
From (\ref{eq:bases}), we have
\begin{equation}\label{eq:same_det}
  \det \Lambda_1 = \frac{1}{\ell} \det \Lambda = \det \Lambda_2.
\end{equation}
Moreover, the multiplication by $a$ map $z \mapsto az$ regarded as an
$\R$--linear endomorphism of $\C$ has determinant $|a|^2$. Indeed, writing
$a = a_0 +DA_1$, the map is represented in the basis $(1,i)$ by the matrix
\[
\begin{pmatrix}
  a_0 & -a_1 \\ a_1 & a_0
\end{pmatrix},
\] whose determinant is $a_0^2 + a_1^2 = |a|^2$.
Next, the
assumption $\Lambda_2 = a \Lambda_1$ together with (\ref{eq:same_det})
give
\[
  \det \Lambda_1 = \det \Lambda_2 = |a|^2 \det \Lambda_1,
\]
which yields $|a|^2=1$.

\medskip

\noindent {\bf Step 3.}  We aim to prove that $a \Lambda =
\Lambda$. Suppose it does not.  Since $\Lambda \subseteq \Lambda_1$,
$\Lambda \subseteq \Lambda_2$ and
$a \Lambda \subseteq a\Lambda_1 = \Lambda_2$, then
$\Lambda + a \Lambda \subseteq \Lambda_2$. Recall that
$\card{\Lambda_2/\Lambda} = \ell$ and $\ell$ is prime. Then, since we
assumed that $\Lambda \varsubsetneq \Lambda + a\Lambda$, we get
$\Lambda + a\Lambda = \Lambda_2$.  Similarly, one deduces that
$a^{-1} \Lambda + \Lambda = \Lambda_1$.  Next, from (\ref{eq:bases}),
we see that $\Lambda_1 + \Lambda_2 = \frac{1}{\ell} \Lambda$ and hence
\[
  a^{-1} \Lambda + \Lambda + a \Lambda = \frac{1}{\ell} \Lambda.
\]
By induction,
\[
  a^{-s}\Lambda + \cdots + a^{-1}\Lambda + \Lambda + a \Lambda +
  \cdots + a^{s}\Lambda = \frac{1}{\ell^{s}}\Lambda.
\]
Therefore, for any $s \geq 0$, there exists a finite sequence
${(\mu^{s}_{i})}_{i=-s}^s$ of elements of $\Lambda$ such that
\[
  \sum_{i=-s}^{s} a^i \mu^{s}_i = \frac{1}{\ell^s} \omega_1. 
\]
Then, for any $N \geq 0$,
\begin{eqnarray*}
  \sum_{s=0}^N \sum_{i=-s}^{s} a^i \mu^{s}_i &=& \sum_{s=0}^N
                                               \frac{1}{\ell^s} \omega_1, \\
  \sum_{i=-N}^{N} a^i \nu_i &=&\sum_{s=0}^N
                                               \frac{1}{\ell^s} \omega_1,
\end{eqnarray*}
where the $\nu_i$'s are in $\Lambda$. When $N$ goes to infinity, the
right hand side is a convergent series. Thus, so does the left hand
side and hence, its general term should go to $0$. From the previous
step, we know that $|a|=1$ and since the $\nu_i$'s are in $\Lambda$
which is discrete, then $\nu_i = 0$ for any sufficiently large $i$.
Therefore, the sequence of partial sums of the left--hand side is
stationary while that of the right--hand side is note.  This is a
contradiction. Therefore $a \Lambda = \Lambda$.
\end{proof}

\begin{lemma}\label{lem:faithful}
  Let $\E_i \eqdef \C / (\Z \oplus \Z i)$ and consider its
  automorphism group $G_i$ induced by the multiplications by
  $\{\pm 1, \pm i\}$ in $\C$.  Then, any $P \in \E_i(\C)$ which has a
  non trivial stabiliser under the action of $G_i$ is in $\E[2]$.

  Similarly, let $\E_\rho \eqdef \C / (\Z \oplus \Z \rho)$, where
  $\rho = e^{\frac{i\pi}{3}}$ with its automorphism group $G_\rho$ induced by
  the multiplications by $\{\pm 1, \pm \rho, \pm \rho^2\}$, then any
  $P \in \E_\rho (\C)$ with a non trivial stabiliser under the action of
  $G_\rho$ is in $\E[6]$.
\end{lemma}

\begin{proof}
  In the case $\E_i$, denote by $\Lambda_i \eqdef \Z \oplus \Z i$.
  Since $G_i$ is cyclic of order $4$, its only nontrivial subgroups
  are $\{\pm 1\}$ and $G_i$ itself. A point $P \in \E_i (\C)$
  stabilised by $\{\pm 1\}$ corresponds to $z \in \C$ such that
  $z \equiv -z \mod \Lambda_i$. That is to say $2z \in \Lambda_i$ and
  hence $P \in \E_i[2]$. Similarly if $P$ is stabilised by all $G_i$
  it is {\em a fortiori} stabilised by $\{\pm 1\}$ and hence should be
  in $\E_i [2]$.
  
  Consider now the case of $\E_\rho$. Denote by $\Lambda_\rho \eqdef
  \Z \oplus \Z \rho$. Since $G_\rho$ is cyclic of
  order $6$, its only possible nontrivial subgroups are $\{\pm 1\}$,
  $\{1, \rho^2, \rho^4\}$ and $G_\rho$ itself. Let us consider points
  which are stabilised by one of these groups.
  
  Let $P \in \E_\rho (\C)$ stabilised by $\{\pm 1\}$, then the very
  same reasoning as for $\E_i$ yields $P \in \E_\rho [2]$.

  Let $P \in \E_\rho(\C)$ stabilised by $\{1, \rho^2, \rho^4\}$. This corresponds
  to $z \in \C$ satisfying $\rho^2 z \equiv z \mod \Lambda_\rho$.
  Writing $z = a + b \rho^2$ for some $a, b \in \R$ and using the relation
  $1+\rho^2 +\rho^4 = 0$, we get
  \[
    a + \rho^2 b \equiv -b + \rho^2(a-b) \mod \Lambda_\rho.
  \]
  This entails that
  \[
    \left\{
      \begin{array}{rcl}
        -b & = & a + \mu \\
        a-b & = & b + \nu,
      \end{array}
    \right.
  \]
  where $\mu, \nu \in \Z$. By elimination, we deduce that $3a \in \Z$
  and $3b \in \Z$, that is to say $z \in \frac{1}{3} \Lambda_\rho$ and
  hence $P \in \E_\rho [3]$.

  Finally, the previous discussion entails that a point stabilised by
  the whole $G_\rho$ should be in $\E_\rho[2] \cap \E_{\rho}[3]$, and hence
  is nothing but $O_{\E_\rho}$.
\end{proof}

\begin{proof}[Proof of Theorem~\ref{thm:genus_X0ell}]
  The idea is to consider the projection map
  $\pi : X_0(\ell) \rightarrow X_0(1)$, which sends a class of
  isomorphisms of isogenies $\phi : \E \rightarrow \E'$ onto the
  isomorphism class of $\E$.  This map is algebraic (this will appear
  more naturally in \S~\ref{ss:mod_eq}).  The objective is to apply
  Riemann Hurwitz formula (Theorem~\ref{thm:RH}) to $\pi$ in order to
  compute the genus of $X_0(\ell)$.

  \medskip

\noindent   {\bf Step 1. The degree of $\pi$.}  The degree of $\pi$ is the
  generic number of pre-images of a point of $X_0(1)$. Such a point
  corresponds to a curve $\E$ up to isomorphism and its pre-image is
  the set of isomorphism classes of $\ell$--isogenies
  $\E \rightarrow \E'$ or equivalently, the ismomorphism classes of
  pairs $(\E, C)$ where $C$ is a subgroup of cardinality $\ell$ of
  $\E [\ell]$.  From Lemma~\ref{lem:isomorphic_subgroups}, if $\E$ has
  no nontrivial automorphism, then two distinct subgroups $C_1, C_2$
  provide non isomorphic pairs $(\E, C_1), (\E, C_2)$. Thus, in this
  situation, the number of pre-images of $\E$ by $\pi$ corresponds to
  the number of subgroups of cardinality $\ell$ in $\E [\ell]$. Since
  $\ell$ is prime, then from Theorem~\ref{thm:torsion_of_elliptic},
  $\E[\ell] \simeq \F_\ell \times \F_\ell$ and hence is a vector space
  of dimension $2$ over $\F_\ell$. Next, a subgroup of cardinality
  $\ell$ of $\E [\ell]$ is nothing but a subspace of dimension $1$ and
  the number of subspaces of dimension $1$ (\ie{} of lines) of
  $\F_\ell \times \F_\ell$ equals $\card{\P^1(\F_{\ell}) = \ell + 1}$.
  Thus,
  \[
    \deg \pi = \ell + 1.
  \]

  Now, the map is ramified at the points corresponding to curves with
  nontrivial automorphisms and possibly the point at inifinity. From
  Theorem~\ref{thm:automorphisms}, the curves with nontrivial
  automorphisms correspond to
  the tori $\C/(\Z \oplus \Z i)$ and $\C/ (\Z \oplus \Z \rho)$, where
  $\rho = e^{\frac{i\pi}{3}}$. For these tori, we need to understand the action
  of the automorphisms on the $\ell$--torsion.

  \medskip
  
  \noindent {\bf Step 2. Ramification at $\C/(\Z \oplus \Z i)$.} The
  curve is equipped with a nontrivial automorphism $\eta$ of order
  $4$, which corresponds to the multiplication by $i$ in $\C$. This
  automorphism acts on the $\ell$--torsion and, from
  \cite[Thm.~III.4.8]{silverman2009book}, such an automorphism is a
  group automorphism and hence its action on $\E[\ell]$ regarded as an
  $\F_\ell$--vector space is $\F_{\ell}$--linear.  We denote by
  $\eta_{\ell}$, the automorphism $\eta$ restricted to $\E [\ell]$.
  From Lemma~\ref{lem:faithful}, since $\ell > 3$ any point in
  $\E [\ell] \setminus \{O_\E\}$ has has an orbit of cardinality $4$
  under $\eta_\ell$. Therefore, $\eta_\ell$ has order $4$ and two
  situations may occur. Either $\ell \equiv 1 \mod 4$, then
  $\F_{\ell}$ contains fourth roots of $1$ and $\eta_\ell$ regarded as
  an $\F_\ell$--automorphism of $\F_\ell \times \F_\ell$ is
  diagonalisable as
  \[
    \begin{pmatrix}
      \iota & 0 \\ 0 & -\iota
    \end{pmatrix},
  \]
  where $\iota$ denotes a primitive fourth root of $1$. In this
  situation, $\eta_\ell$ acts on the lines of $\F_\ell \times \F_\ell$
  by fixing the two lines corresponding to the eigenspaces of
  $\eta_\ell$ and any other line has an orbit of cardinality $2$,
  indeed $\eta^2_{\ell} = -\text{Id}$, which leaves any line
  invariant.
  Two lines of $\E [\ell]$ in a same orbit under $\eta_\ell$
  correspond to a same point in $X_0(\ell)$. This point is a ramification
  point with ramification index $2$.
  Therefore, if $\ell \equiv 1 \mod 4$, then there are $2$ unramified
  points in the pre-image of the isomorphism class of $\E$ by $\pi$
  and $\frac{\ell - 1}{2}$ ramified points with ramification index $2$.

  Otherwise $\ell \equiv 3 \mod 4$. In this situation $\eta_\ell$ has
  no eigenspace in $\E[\ell]$ and the orbit of any line has
  cardinality $2$. Thus, there are $\frac{\ell + 1}2$ points above
  $\E$ which are all ramified with ramification index $2$.

  \medskip

  \noindent {\bf Step 3. Ramification at $\C/ (\Z \oplus \Z \rho)$.}
  Here we have an automorphism $\eta$ of order $6$ and denote again by
  $\eta_\ell$ its restriction to $\E [\ell]$. Here again, from
  Lemma~\ref{lem:faithful}, we know that $\eta_\ell$ has also order
  $6$.  In this situation, if $\ell \equiv 1 \mod 3$, then $\F_\ell$
  contains sixth roots of unity and $\eta_\ell$ is diagonalisable.
  Therefore, the two eigenspaces of $\eta_\ell$ are left invariant and
  any other $\F_\ell$--line of $\E[\ell]$ has an orbit of cardinality
  $3$. Indeed, here again $\rho^3 = - \text{Id}$ and hence leaves any line
  globally invariant.  In such a situation, the pre-image of $\E$
  consists in $2$ unramified points corresponding to the two
  eigenspaces of $\eta_\ell$ in $\E [\ell]$ and $\frac{\ell - 1}{3}$
  points with ramification index $3$.

  If $\ell \equiv 2 \mod 3$, then any line of $\E [\ell]$ has an orbit
  of cardinality $3$ under the action of $\eta_\ell$ and hence the
  pre-image of $\E$ by $\pi$ consists in $\frac{\ell + 1}{3}$ points,
  all with ramification index $3$.

  \medskip

\noindent  {\bf Step 4. Ramification at infinity.}
  The point at infinity of $X_0(1)$ is the quotient of $\P^1(\Q)$
  under the action of $\SL$, which, from Proposition~\ref{prop:cusps},
  consists in a single orbit.
  We wish to estimate the number of orbits in $\P^1(\Q)$ under the
  action of $\Gamma_0(\ell)$. We claim that their number is $2$, namely,
  the orbit of $(0:1)$ and that of $(1:0)$.
  Indeed,
  \[
    \Gamma_0(\ell) \cdot (0:1) = \{(b:d) \in \P^1(\Q) \ \text{with}\
    \gcd(b,d) = 1 \ \text{and}\ d\ \text{prime to }\ell\}
  \]
  and
  \[
    \Gamma_0(\ell) \cdot (1:0) = \{(a:c) \in \P^1(\Q) \ \text{with}\
    \gcd(a,c) = 1 \ \text{and}\ \ell\ \text{dividing } c\}.
  \]
  One easily sees that the two orbits form a partition of $\P^1(\Q)$. This
  entails that the pre-image of the point at infinity of $X_0(1)$
  consists in two points $P,Q$ whose ramification indexes satisfy
  $e_P + e_Q = \ell + 1$.

  \medskip

\noindent  {\bf Final step. Computation of the genus.}
  Denote by $g_\ell$ the genus of $X_0(\ell)$ and by $g_1 = 0$ that of $X_0(1).$
  Riemann--Hurwitz formula asserts that
  \[
   2g_\ell - 2 = (2 g_1 - 2)(\ell + 1) + \nu_i + \nu_\rho + \nu_{\infty}, 
 \]
 where $\nu_i,\nu_\rho$ and $\nu_{\infty}$ are the respective contributions
 of the ramifications above $\C/(\Z \oplus \Z i)$, $\C/(\Z \oplus\Z \rho)$
 and the point at infinity.
 We get
 \[
   \nu_2 = \left\{
     \begin{array}{ccl}
       \frac{\ell-1}{2} & \text{if} & \ell \equiv 1 \mod 4\\
       \frac{\ell + 1}{2} & \text{if} & \ell \equiv 3 \mod 4
     \end{array}
   \right.,
   \quad
      \nu_3 = \left\{
     \begin{array}{ccl}
       2\frac{\ell-1}{3} & \text{if} & \ell \equiv 1 \mod 3\\
       2\frac{\ell + 1}{3} & \text{if} & \ell \equiv 2 \mod 3
     \end{array}
   \right.
 \quad
 \text{and}
 \quad
   \nu_{\infty} = \ell - 1.
 \]
 An easy but cumbersome calculation treating separately
 the four cases $\ell \equiv 1, 5, 7, 11 \mod 12$ yields the expected result.
\end{proof}

\subsection{The modular equation}\label{ss:mod_eq}
To conclude this section, we give a statement whose proof is omitted
but which may help the reader to be convinced that $X_0(\ell)$ has a
structure of algebraic curve. We refer the reader to
{\cite[Thm.~6.1]{milneMF}} for a proof.

\begin{theorem}
  There exists an irreducible polynomial $\Phi_\ell \in \Z [x,y]$ such
  that for any pair $\E, \E'$ of elliptic curves related with a degree
  $\ell$ isogeny $\E \rightarrow \E'$, then
  $\Phi_\ell(j(\E), j(\E')) = 0$.
\end{theorem}

\begin{remark}
  A database of the polynomials $\Phi_\ell$ for small values of $\ell$
  is available on Andrew Sutherland's webpage: 
  \href{https://math.mit.edu/~drew/ClassicalModPolys.html}{\tt https://math.mit.edu/$\sim$drew/ClassicalModPolys.html}
\end{remark}

Let us give some comments about this statement. First, note that the
projective closure of the complex curve of equation $\Phi_{\ell}(x,y)=0$ is a
``singular model'' for $X_0(\ell)$. Indeed, any point of $X_0(\ell)$
corresponds to an isomorphism class of $\ell$--isogeny $\E \rightarrow \E'$.
This yields a rational map
\[
  \map {X_0(\ell)}
  {\P^2(\C)}
  {(\E \rightarrow \E')}
  {(j(\E):j(\E'):1).}
\]
The image of this map is contained into the curve with equation
$\Phi_\ell (x,y) = 0$. However, this latter curve is full of
singularities and hence is not isomorphic to $X_0(\ell)$.
Nevertheless, (and this is far from being obvious) this permits to
deduce that $X_0(\ell)$ is itself defined over $\Q$ and hence, its
reduction modulo $p$ makes sense.

An interesting fact is that, since $\Phi_\ell \in \Z[x,y]$, for any
pair of $\ell$--isogenous curves $\E \rightarrow \E'$ over $\F_p$ we
have $\Phi_\ell (j(\E), j(\E')) \equiv 0 \mod p$.  Moreover, if $\ell$
and $p$ are prime to each other, it is known that the polynomial
$\Phi_\ell$ is irreducible modulo $p$ (see for instance
\cite[Thm.~5.9]{moreno1990book}). Thus, the curve over $\F_p$ of equation
$\Phi_\ell (x,y) = 0$ turns out to be a singular model of a smooth
curve over $\F_p$, that we will also denote by $X_0(\ell)$ such that
$X_0(\ell)(\overline{\F}_p)$ parameterises $\ell$--isogenies
$\E \rightarrow \E'$ over $\overline{\F}_p$ up to isomorphism.  These
curves over $\F_p$ will be the objects of interest in order to prove
the main theorem of this course, namely Theorem~\ref{thm:main}.

Finally, the reader interested in a rigorous study of the reductions
of modular curves cannot avoid the language of schemes. For such
a development, we refer the reader to the article of Celgene and Rapoport
\cite{deligne1973} or the book of Katz and Mazur \cite{katz1985book}.

 \section{Proof of the main Theorem}\label{sec:proof_of_main}
Now, we almost have the material to prove Theorem~\ref{thm:main}.  We
have our family of curves $X_0(\ell)$ for $\ell$ a prime integer
distinct from the characteristic $p$.

\subsection{Genus of modular curves over finite fields}
Let us briefly discuss the genus of the curve. The discussion to
follow is far from being trivial. Thus, the reader is encouraged first
to directly admit the conclusion. Namely that the genus of a modular
curve over a finite field is that of its complex counterpart.  Let us
briefly sketch the reasons why this holds.

It is known (see for instance in \cite[Thm.~5.9]{moreno1990book})
that, the curve $X_0(\ell)$ has a smooth projective model described by
equation with coefficients in $\Z$ and whose reduction modulo $p$ is
smooth too.

Next, as already mentioned in Remark~\ref{rem:arith_genus}, two different
notions of genus are associated to a curve, the {\em arithmetic} genus
$p_a$ and the {\em geometric} one $g$. The genus introduced by
Definition~\ref{def:genus} in \S~\ref{ss:genus_and_RR} is the
geometric one. The arithmetic genus, which can be defined for instance
from the Hilbert function of the variety
\cite[Ch.~IV]{hartshorne1977book}, is always larger than or equal to
the geometric one and they coincide if and only if the curve is
smooth.

Next, Grauert Theorem \cite[Cor.~III.12.9]{hartshorne1977book} permits
to assert that the reduction modulo $p$ of the aforementioned model
$X_0(\ell)$ has the same arithmetic genus as the complex curve
itself. Moreover, since this model and its reduction are smooth, they
also have the same geometric genus.  Therefore, the genus of the curve
$X_0(\ell)$ over $\F_p$ is the same as that of its complex counterpart
and hence is given by Theorem~\ref{thm:genus_X0ell}.

\subsection{The locus of supersingular curves}
There remains to get an estimate of the
number of rational points of such curves.  For this we will focus on
$\F_{p^2}$--points since their number can be bounded from below using
the two following statements.

\begin{proposition}
  Let $\E$ be a supersingular elliptic curve over $\overline{\F}_p$,
  then $\E$ is defined over $\F_{p^2}$.
\end{proposition}

\begin{proof}
  By definition, a supersingular curve $\E$ satisfies $\E[p] = \{0\}$.
  Therefore, the multiplication by $p$ map $[p] : E \rightarrow E$ is
  totally inseparable.

  Consider now the Frobenius map
  \[
    \phi : \map{\E}{\E^{(p)}}{(x,y)}{(x^p,y^p).}
  \]
  It is a degree $p$ isogeny, hence it has a dual isogeny $\hat \phi$
  such that
  $\hat \phi \circ \phi = [p]$. Since $[p]$ is totally inseparable, $\hat \phi$
  should be inseparable either and hence, so should be the Frobenius map
  \[
    \hat \phi : \map{\E^{(p)}}{\E^{(p^2)}}{(x, y)}{(x^p,y^p).}
  \]
  Thus, $\E^{(p^2)} = \E$ and hence $\E$ is defined over $\F_{p^2}$.
\end{proof}

\begin{theorem}\label{thm:supsersing_classes}
  The number of $\overline{\F}_p$--isomorphism classes of supersingular
  elliptic curves over $\overline{\F}_p$ equals
  \[
    \left \lfloor \frac{p}{12} \right \rfloor + \left\{
      \begin{array}{ccccc}
        0 & \text{if} & p & \equiv & 1 \mod 12 \\
        1 & \text{if} & p & \equiv & 5 \mod 12 \\
        1 & \text{if} & p & \equiv & 7 \mod 12 \\
        2 & \text{if} & p & \equiv & 11 \mod 12.
      \end{array}
      \right.
  \]
\end{theorem}

\begin{proof}
  See \cite[Thm.~V.4.1]{silverman2009book}.
\end{proof}

\subsection{Proof of the main theorem}
With these two last statements at hand we can finally provide the
proof of Theorem~\ref{thm:main}. We restrict the proof to the case
$p \geq 5$. Note that Tsfasman--Vl\u{a}du\c{t}--Zink theorem remains
true when $p = 2,3$ but the coding theoretic interest is rather
limited.

\begin{proof}[Proof of Theorem~\ref{thm:main}]
  Consider the sequence of curves $X_0(\ell)$ for
  $\ell \equiv 11 \mod 12$.  From Theorem~\ref{thm:genus_X0ell} it has
  genus $g_\ell = \frac{\ell + 1}{12}$.  From
  Theorem~\ref{thm:supsersing_classes}, the curve $X_0(1)$ has at
  least $\frac{p-1}{12}$ $\F_{p^2}$--rational points corresponding to
  isomorphism classes of supersingular elliptic curves.  Such an
  elliptic curve with no nontrivial automorphism has $\ell + 1$
  pre-images in $X_0(\ell)$ which also correspond to supersingular
  elliptic curves and hence are $\F_{p^2}$--rational points.
  Depending on the class of $p$ modulo $12$ the curves with
  $j$--invariant $0$ and $1728$ may be supersingular. More precisely,
  from \cite[Ex.~V.4.4 \& V.4.5]{silverman2009book},
  \begin{itemize}
  \item[\textbullet] for $p \equiv 1 \mod 12$, both curves are
    ordinary (\ie{} non supersingular) and then any supersingular
    curve has no nontrivial automorphism and hence has $\ell + 1$
    pre-images in $X_0(\ell)(\F_{p^2})$. Therefore, from
    Theorem~\ref{thm:supsersing_classes},
    \[
      \card{X_0(\ell)(\F_{p^2})} \geq (\ell + 1) \frac{p-1}{12}\cdot
    \]
  \item[\textbullet] for $p \equiv 5 \mod 12$, the curve with $j=0$ is
    supersingular and the one with $j = 1728$ is ordinary. Therefore,
    there are $\frac{p-5}{12}+1$ supersingular curves and all of them
    but one have $\ell + 1$ distinct pre-images.  The remaining curve
    is the one with $j = 0$ and an automorphism group of order $6$.
    Its treatment is very similar to the proof of
    Theorem~\ref{thm:genus_X0ell}. Consider the action of the
    automorphism of order $6$ on the $\ell$--torsion.  From
    Lemma~\ref{lem:faithful}, this induces an automorphism $\eta$ or
    order $6$ of $\E [\ell] \simeq \F_\ell^2$.  Since
    $\ell \equiv 11 \mod 12$, then $\ell \equiv 2 \mod 3$ and hence
    $\F_\ell$ does not contains the sixth roots of unity. Thus, $\eta$
    is not diagonalisable and hence cannot fix a line. Consequently,
    one proves that the orbit of any line is the union of $3$ distinct
    lines ($\eta^3 = -\text{Id}$, which fixes the lines).  Therefore,
    the curve with $j$--invariant $0$ has $\frac{\ell + 1}{3}$
    pre-images and Consequently, using
    Theorem~\ref{thm:supsersing_classes}:
    \[
      \card{X_0(\ell)(\F_{p^2})} \geq (\ell + 1)\frac{p-5}{12} +
      \frac{\ell + 1}{3} = (\ell + 1)\frac{p-1}{12}\cdot
    \]
  \item[\textbullet] for $p \equiv 7 \mod 12$, the curve with $j=0$ is
    ordinary and the one with $j = 1728$ is supersingular. A similar reasoning
    yields
    \[
      \card{X_0(\ell)(\F_{p^2})} \geq (\ell + 1)\frac{p-7}{12} +
      \frac{\ell + 1}{2} = (\ell + 1)\frac{p-1}{12}\cdot
    \]
  \item[\textbullet] for $p \equiv 11 \mod 12$, both curves are
    supersingular and we get:
    \[
      \card{X_0(\ell)(\F_{p^2})} \geq (\ell + 1)\frac{p-11}{12} +
      \frac{\ell + 1}{2} + \frac{\ell + 1}{3} = (\ell + 1)\frac{p-1}{12}\cdot
    \]
  \end{itemize}
  In summary, we always have a lower bound $(\ell + 1)\frac{p-1}{12}$
  on the number of rational points. Then
  \[
    \frac{\card{X_0(\ell)(\F_{p^2})}}{g_\ell} \geq
    \frac{\frac{p-1}{12}(\ell + 1)}{\left( \frac{\ell + 1}{12}\right)}
    = p-1.
  \]
  Thus, over $\Fq$ for $q = p^2$, we identified a family of curves
  whose number of $\Fq$--rational points goes to infinity and whose ratio,
  number of $\Fq$--points divided by the genus goes to $\sqrt{q}-1$.
  Which turns out to be optimal.
\end{proof}

\bibliographystyle{alpha}

\end{document}